\documentclass[12pt]{extarticle}
\usepackage{amsmath, amsthm, amssymb, color}
\usepackage[colorlinks=true,linkcolor=blue,urlcolor=blue,citecolor=teal]{hyperref}
\usepackage{graphicx}
\usepackage{caption}
\usepackage{mathtools}
\usepackage{enumitem}
\usepackage{verbatim}
\usepackage{tikz}
\usetikzlibrary{matrix}
\usetikzlibrary{arrows}
\usepackage{algorithm}
\usepackage[noend]{algpseudocode}
\usepackage{caption}
\usepackage[normalem]{ulem}
\usepackage{subcaption}
\oddsidemargin \evensidemargin
\usepackage{makecell}
\usepackage{multirow,array}

\newtheorem{theorem}{Theorem}
\newtheorem{proposition}[theorem]{Proposition}
\newtheorem{lemma}[theorem]{Lemma}

\theoremstyle{definition}
\newtheorem{definition}[theorem]{Definition}

\newtheorem{remark}[theorem]{Remark}
\newtheorem{conjecture}[theorem]{Conjecture}

\newtheorem{example}[theorem]{Example}

\newcommand{\PP}{\mathbb{P}}
\newcommand{\RR}{\mathbb{R}}

\newcommand{\CC}{\mathbb{C} }

\title{\bf Geometry of Dependency Equilibria}
\author{Irem Portakal and Bernd Sturmfels}

\date{}

\begin{document}
\maketitle

\begin{abstract} \noindent
An $n$-person game is specified by $n$ tensors of the same format. We view
 its equilibria as points in that tensor space.
Dependency equilibria are defined by
linear constraints on conditional probabilities, and thus by determinantal
quadrics in the tensor entries. These equations cut out the
Spohn variety, named after
the philosopher who introduced dependency equilibria.
The Nash equilibria among these are the tensors of rank one.
We study the real algebraic geometry of the Spohn variety.
This variety is rational, except for $2 \times 2$ games, when it is
an elliptic curve. For
$3 \times 2$ games, it is a del Pezzo surface of degree two.
We characterize the payoff regions and their boundaries
using oriented matroids,
and we develop the connection  to 
Bayesian networks in statistics.
\end{abstract}

\section{Introduction}
\label{sec1}

The geometry of Nash equilibria has been a topic of considerable
interest in economics, mathematics and computer science. It is known,
thanks to Datta's Universality Theorem \cite{datta}, that the
set of Nash equilibria can be an essentially arbitrary semialgebraic set.
Yet, a game with generic payoff tables has only finitely many
Nash equilibria, with tight bounds known for their number \cite{MCKELVEY1997411}.
They can be found with the tools of computational algebraic~geometry.

For many games one encounters Nash equilibria
with undesirable or counterintuitive properties.
This issue has been a concern not just in the economics literature, but also 
in philosophy.  
Several authors proposed more inclusive notions of
equilibria. One of these is the concept 
of {\em correlated equilibria}, due to Aumann~\cite{aumann2}. In this concept,  one augments the original game with a coordination device which allows players to coordinate actions (i.e. \ a joint probability distribution).
These equilibria form a convex polytope in tensor space,
studied in \cite{brandenburgportakal, NCH},
with Nash equilibria being precisely the rank one tensors.

In this article, we examine another inclusive notion of equilibria,
introduced by a prominent philosopher, Wolfgang Spohn, in his 
articles \cite{homo, spohn}. Spohn's notion of dependency equilibria
leads to interesting structures in nonlinear algebra \cite{MS}.
Unlike the polyhedral setting of correlated equilibria, the characterization
of dependency equilibria requires nonlinear polynomials,
even for two-player games. This is the reason why they are interesting for us.

Spohn offers the following warning about the nonlinear algebra that arises in his approach: {\em
The computation of dependency equilibria seems to be a messy business. Obviously it requires one to solve quadratic equations in two-person games, and the more persons, the higher the order of the polynomials we become entangled with. All linear ease is lost. Therefore, I cannot offer a well developed theory of dependency equilibria} \cite[page 779, Section 3]{spohn}.

\smallskip

This paper
 lays the foundations for the desired theory, by introducing
novel algebraic varieties in tensor spaces.
The bemoaned loss of linear ease is our journey's point of departure.

It is useful to think of this article as a case study in algebraic statistics  \cite{BC, Sul}.
In that field one examines statistical models for $n$ discrete random variables.
Such a model is a semialgebraic set whose points are positive tensors
whose entries sum to one. These represent joint probability distributions,
and the statistical task is to identify points that best explain some given data set.
To address such an optimization problem, it is advantageous to 
relax the constraint that tensors are real and positive. Thus,
one replaces the model by its Zariski closure in a complex projective space,
and one studies algebro-geometric features -- such as
dimension, degree, equations, decomposition, and singularities -- of 
these varieties. 

The statistical model in this article is the set of dependency equilibria 
of an $n$-person game in normal form. These equilibria are 
real positive tensors whose entries sum to one.
Relaxing the reality constraints yields an algebraic
variety in complex projective space. This is called the {\em Spohn variety}
of the game,  in recognition of the fundamental work in \cite{homo, spohn}.

Our presentation is organized as follows.
In Section \ref{sec2} we review the basics on
 $n$-player games in normal form, and  we present 
the equations that define dependency equilibria.
After clearing denominators, these are
expressed as the $2 \times 2$ minors of $n$
matrices whose entries are linear forms in the entries of $P$.
Small cases are worked out in Examples \ref{ex:bach},
 \ref{ex: dependency equilibria of 2x2 game},
 \ref{ex:twotwotwo} and  \ref{ex:centipede}.
 
The Spohn variety $\mathcal{V}_X$ of a normal form game $X$ is formally introduced
in Section \ref{sec3}. We determine its dimension and degree in
Theorem \ref{thm:spohn}. The intersection of $\mathcal{V}_X$ with the Segre variety recovers
the Nash equilibria. Theorem \ref{thm:rational} shows that
$\mathcal{V}_X $ is generally rational, with
an explicit  rational parametrization.
 Example~\ref{ex:delpezzo}
covers Del Pezzo surfaces of degree~two.

Section \ref{sec4} offers a detailed study of the dependency equilibria for
$2 \times 2$ matrix games. This case is an exceptional case
because the Spohn variety $\mathcal{V}_X$ is not rational.
It is the intersection of two quadrics in $\PP^3$, hence
an elliptic curve, when the payoff matrices are generic.
A formula for the j-invariant  is given in Proposition  \ref{prop:jinv}. The
real picture is determined in
Theorem~\ref{thm:arcs}.

Section \ref{sec5} concerns the payoff region $\mathcal{P}_X$.
This is a semialgebraic subset of $\RR^n$,
visualized in Figures \ref{fig: payoff curve} and \ref{fig: payoff region}.
The points of $\mathcal{P}_X$ are the expected utilities of positive points on $\mathcal{V}_X$.
Theorem \ref{thm:OM} identifies that region 
in the oriented matroid stratification given by the Konstanz matrix $K_X(x)$.
Its algebraic boundaries are determinantal hypersurfaces,
such as the K3 surfaces in Example \ref{ex:222Konstanz}.
 These offer  an algebraic representation for
Pareto optimal equilibria.

 Section \ref{sec6} develops a perspective 
that offers dimensionality reduction 
and a connection to data analysis.
Namely, we consider conditional independence models,
in the sense of algebraic statistics \cite{BC, Sul}. These models are
represented by projective varieties. We focus on the
case of Bayesian networks \cite{GSS}. Their importance for
dependency equilibria was
already envisioned by Spohn in \cite[Section 3]{homo}.
This offers many opportunities for future research.

\section{Games, Tensors and Equilibria}
\label{sec2}

We work in the setting of normal form games, using the notation fixed in \cite[Section~6.3]{CBMS}.
Our game has $n$ players, labeled as $1,2,\ldots,n$. The $i$th player can select from $d_i$ pure
strategies. This set of pure strategies is taken to be $[d_i] = \{1,2,\ldots,d_i\}$. The game is specified by $n$ payoff tables
$X^{(1)}, X^{(2)}, \ldots, X^{(n)}$. Each payoff table is a tensor
of format $d_1 \times d_2 \times \cdots \times d_n$ whose entries are arbitrary real numbers.
The entry $X^{(i)}_{j_1 j_2 \cdots j_n} \in \RR$ represents the payoff for player $i$
if player $1$ chooses pure strategy $j_1$, player~$2$ chooses pure strategy $j_2$, etc.
These choices are to be understood probabilistically.
Think of the $n$ players as random variables. The $i$th random variable
has the state space $[d_i]$.
The players collectively choose a mixed strategy, which is a joint
probability distribution $P$.
More precisely, $P$ is a tensor of format
  $d_1 \times d_2 \times \cdots \times d_n$    whose entries
are positive reals that sum to $1$. The entry
$p_{j_1 j_2 \cdots j_n}$ is the probability that
player $1$ chooses pure strategy $j_1$, player~$2$ chooses pure strategy $j_2$, etc.

We write $V = \RR^{d_1 \times d_2 \times \cdots \times d_n}$ for the
real vector space of all tensors. Let $\PP(V)$ denote the corresponding
projective space, and let $\Delta$ be the open simplex of positive real points in $\PP(V)$.
The set of equilibria of our game is a subset of $\Delta$, and we are interested
in its Zariski closure in $\PP(V)$. The classical theory of Nash equilibria arises through the
{\em Segre variety} $\PP^{d_1-1} \times \PP^{d_2-1} \times \cdots \times \PP^{d_n-1}$
whose points are the tensors of rank one in $\PP(V)$. Namely,
the entries of a rank one tensor $P$ factor into the decision variables of \cite[Section 6.3]{CBMS} as follows:
$$ p_{j_1 j_2 \cdots j_n} \,= \, \pi^{(1)}_{j_1} \cdot \pi^{(2)}_{j_2} \cdot\,\ldots \, \cdot \pi^{(n)}_{j_n} .$$
Here $\pi^{(i)}_{j_i}$ represents the probability that player $i$ unilaterally selects pure strategy $j_i$.
In the study of totally mixed Nash equilibria, these quantities are positive reals, and they satisfy
\begin{equation}
\label{eq:sumtoone} \pi^{(i)}_{1} + \pi^{(i)}_{2} + \cdots + \pi^{(i)}_{d_i} \,   \, = \,\, 1 . 
\end{equation}
However, in what follows the $n$ players  are not independent. 
We view them as acting~together.
Their collective choice of a mixed strategy is thus a tensor $P$ which need not have rank~$1$.

Consider two players with binary choices, so
$n=d_1=d_2=2$.
Here $V = \RR^{2 \times 2}$ is a four-dimensional vector space,
and $P(V) = \PP^3$ is the projective space  whose points are $2 \times 2$ matrices up to scaling.
A game is specified by  two matrices $X^{(1)}$ and $ X^{(2)}$ in $V$.
The two players collectively choose a joint probability distribution $P$ for two
binary random variables.
Thus, they choose a positive matrix
 $P = \begin{small} \begin{bmatrix} p_{11} \! & \! p_{12} \\ p_{21} \! & \! p_{22} \end{bmatrix} \end{small}$
 whose entries sum to one, i.e.~$P \in \Delta$.
 
 \begin{example}[Bach or Stravinsky] 
 \label{ex:bach} A couple decides which of two  concerts to attend.
   The payoff matrices 
 indicate their preferences among composers, ${\rm Bach}$ = $1$ or ${\rm Stravinsky}$ = $2$:
 \begin{equation}
 \label{eq:bach}  X^{(1)} =   \begin{bmatrix} 3 & 0 \\ 0 & 2   \end{bmatrix}
 \quad {\rm and} \quad
 X^{(2)} =   \begin{bmatrix} 2 & 0 \\ 0 & 3   \end{bmatrix}.
\end{equation}
In texts on game theory,  this is called a bimatrix game.
The two payoff matrices are often written in a combined table.
For the game (\ref{eq:bach}), the combined table looks as follows:
 \begin{table}[H] \vspace{-0.1in}
\centering
    \setlength{\extrarowheight}{2pt}
    \begin{tabular}{cc|c|c|}
     & \multicolumn{1}{c}{} & \multicolumn{2}{c}{Player $2$}\\
      & \multicolumn{1}{c}{} & \multicolumn{1}{c}{Bach}  & \multicolumn{1}{c}{Stravinsky} \\\cline{3-4}
      \multirow{2}*{Player $1$} 
       & Bach & $(3,2)$ & $(0,0)$ \\\cline{3-4}
      & Stravinsky & $(0,0)$ & $(2,3)$ \\\cline{3-4}
    \end{tabular}
     \end{table} 
\noindent        Different entries are used in
     \cite[Section 3]{spohn}. We refer to that source for further examples.
          The four pure choices BB, BS, SB and SS label
the vertices of the tetrahedron in Figure \ref{fig: dependency equilibria}.
In the game, the couple selects a mixed strategy $P$, which is a point in that tetrahedron.
      \hfill $\diamond$
  \end{example}

Returning to our general set-up, we consider the
expected payoff for the $i$th player. By definition, this is the
dot product of the tensors  $X^{(i)}$ and $P$. In symbols, the 
{\em expected payoff}~is
\begin{equation}
\label{eq:expected}
\begin{small} P X^{(i)} \quad = \quad
 \sum_{j_1=1}^{d_1} \sum_{j_2=1}^{d_2} \cdots \sum_{j_n=1}^{d_n}
p_{j_1 j_2 \cdots j_n} X^{(i)}_{j_1 j_2 \cdots j_n}. \end{small}
\end{equation}
Player $i$ desires this quantity to be as large of possible. 
Aumann's correlated equilibria \cite{aumann2} are choices of $P$ where no player can raise their expected payoff by changing their strategy or breaking their part of the agreed joint probability distribution while assuming that the other players adhere to their own recommendations. The set of correlated equilibria
is a convex polytope inside the simplex $\Delta$. 
Its combinatorial structure is studied in~\cite{brandenburgportakal, NCH}.

In Spohn's theory \cite{spohn},  expected payoff is replaced by conditional expected payoff.
We focus on the payoff expected by player $i$ conditioned on her having fixed
pure strategy~$k \in [d_i]$. In precise mathematical terms,
the {\em conditional expected payoff}  is the ratio of two linear forms in the entries of $P$, each of which has $d_1 \cdots d_{i-1} d_{i+1} \cdots d_n$ summands.
The numerator is the subsum of (\ref{eq:expected}) given by all
summands with $j_i = k$. The denominator is the sum of all probabilities
$p_{j_1 j_2 \cdots j_n} $ where $j_i = k$. In algebraic statistics texts, this is
denoted $p_{+\cdots + k + \cdots+}$. 

Here is now the key definition due to Spohn \cite{homo, spohn}.
Consider the game given by the tuple
 $X = (X^{(1)}, X^{(2)},\ldots,X^{(n)})$.
 A tensor $P$ in $\Delta$ is
a {\em dependency equilibrium} for $X$  if the
conditional expected payoff of each player $i$
does not depend on her choice $k$. In symbols,
this definition says that the following equations hold, 
 for all $i \in [n]$ and all  $k,k' \in [d_i]$:
\begin{equation}\label{eq: dependency equilibria}
\begin{small}
\sum_{j_1 = 1}^{d_1} \cdots \widehat{\sum_{j_{i} = 1}^{d_{i}}} \cdots \sum_{j_n = 1}^{d_n} X^{(i)}_{j_1 \cdots  k  \cdots j_n} \frac{p_{j_1 \cdots  k  \cdots j_n}}{p_{+\cdots + k + \cdots+}} 
\quad = \quad  \sum_{j_1 = 1}^{d_1} \cdots \widehat{\sum_{j_{i} = 1}^{d_{i}}} \cdots \sum_{j_n = 1}^{d_n} X^{(i)}_{j_1 \cdots  k' \cdots j_n} \frac{p_{j_1 \cdots k'  \cdots j_n}}{p_{+\cdots + k' + \cdots+}}.
\end{small}
\end{equation}
Thus, dependency equilibria are defined by
certain equalities among ratios of linear forms. 

One issue with this definition is that 
$p_{+\cdots + k + \cdots+}$ might be zero.
Spohn  calls this a ``technical flaw'' \cite[Section 2]{spohn}, and he suggests a fix
by taking limits to the boundary of $\Delta$.
From the algebraic statistics perspective, this is not a flaw but a feature.
Many models are defined by constraints on strictly positive probabilities.
Possible extensions to the boundary are studied using
the technique of primary decomposition \cite[Section 4.3]{Sul}.
We here disregard boundary phenomena since $\Delta$ is the open simplex.
 This allows us to 
divide by $p_{+\cdots + k + \cdots+}$.

We have argued that clearing denominators in  (\ref{eq: dependency equilibria}) 
does not change the solution sets of interest.
 Thus we can write our equations as $2 \times 2$ determinants of
 linear forms in the entries of~$P$.
  We define a matrix $M_i = M_i(P)$ with $d_i$ rows and two columns
 as follows. The $k$th row of $M_i$ 
 consists of the denominator and the numerator of the ratio on 
 the left of  (\ref{eq: dependency equilibria}):
 \begin{equation}\label{eq: Spohn matrices}
M_i \,=\, M_i(P) \,\,:= \,\,\,\begin{bmatrix}
\vdots & \vdots \\
\,\,p_{+\cdots + k + \cdots+}\, &\,\, \sum_{j_1 = 1}^{d_1} \cdots \widehat{\sum_{j_{i} = 1}^{d_{i}}} \cdots \sum_{j_n = 1}^{d_n} X^{(i)}_{j_1 \cdots  k  \cdots j_n} p_{j_1 \cdots  k  \cdots j_n}\, \\
\vdots & \vdots 
\end{bmatrix} \! .
\end{equation}
Dependency equilibria for $X$ are the points $P \in \Delta$
for which each $M_i$ has rank one.
If $n$ is small then we simplify our notation by using letters $a,b,c$ for  the tensors
$X^{(1)},X^{(2)},X^{(3)}$.

\begin{example}[$2\times2$ games]\label{ex: dependency equilibria of 2x2 game}
Let $n=d_1 = d_2=2$ and $a_{ij},b_{ij} \in \RR$.
 The matrices in (\ref{eq: Spohn matrices}) are
$$ M_1 \,=\, \begin{bmatrix}
\,p_{11} + p_{12} &\, a_{11} p_{11} +a_{12} p_{12} \\
\,p_{21} + p_{22}  & \,a_{21} p_{21} +a_{22} p_{22} 
\end{bmatrix} \quad {\rm and}  \quad M_2 \,=\,  \begin{bmatrix}
\,p_{11} + p_{21} & \,b_{11} p_{11}  +b_{21} p_{21} \\
\, p_{12} + p_{22} & \,b_{12} p_{12} +b_{22} p_{22} \\
\end{bmatrix}. $$
The dependency equilibria are solutions in $\Delta$ to the equations
 ${\rm det}(M_1) = {\rm det}(M_2) = 0$.  
  \hfill $\diamond$
\end{example}

\begin{example}[$2\times2\times2$ games] \label{ex:twotwotwo}
Consider a game with three
players who have binary choices, i.e.~$n=3$ and $d_1=d_2=d_3=2$.
In \cite[Section 6.1]{CBMS} the players are called Adam, Bob and Carl,
and their payoff tables are
$X^{(1)} = (a_{ijk})$, $X^{(2)} = (b_{ijk})$ and $ X^{(3)} = (c_{ijk})$.
Dependency equilibria are  $2 \times 2 \times 2$
tensors $P = (p_{ijk})$ such that
these three $2 \times 2$ matrices have rank~$\leq 1$:
$$ M_1 \,\, = \,\,  \begin{bmatrix}
 p_{111} + p_{112} + p_{121} + p_{122} & \quad   a_{111} p_{111} + a_{112} p_{112} + a_{121} p_{121} + a_{122} p_{122} \\
 p_{211} + p_{212} + p_{221} + p_{222} & \quad  a_{211} p_{211} + a_{212} p_{212} + a_{221} p_{221} + a_{222} p_{222} \end{bmatrix},$$
 $$
M_2 \,\,= \,\, \begin{bmatrix} 
 p_{111} + p_{112} + p_{211} + p_{212} & \quad   b_{111} p_{111} + b_{112} p_{112} + b_{211} p_{211} + b_{212} p_{212} \\
 p_{121} + p_{122} + p_{221} + p_{222} &  \quad b_{121} p_{121} + b_{122} p_{122} + b_{221} p_{221} + b_{222} p_{222]} \end{bmatrix},
 $$
 $$
M_3 \,\, = \,\, \begin{bmatrix}
 p_{111} + p_{121} + p_{211} + p_{221} & \quad   c_{111} p_{111} + c_{121} p_{121} + c_{211} p_{211} + c_{221} p_{221} \\
p_{112} + p_{122} + p_{212} + p_{222} & \quad   c_{112} p_{112} + c_{122} p_{122} + c_{212} p_{212} + c_{222} p_{222}
\end{bmatrix}.
$$
If $X=(A,B,C)$  is generic then their determinants
are quadrics that intersect transversally.
This defines an irreducible variety $\mathcal{V}_X$ 
of dimension $4$ and degree $8$ in the tensor space $\PP^7$.
We now intersect $\mathcal{V}_X$ with the Segre variety
$\PP^1 \times \PP^1 \times \PP^1$ of rank one tensors in $\PP^7$. 
Setting $\alpha = \pi^{(1)}_1$, $\beta=\pi^{(2)}_1$, and $\gamma=\pi^{(3)}_1$, we use
the following parametrization for the Segre variety:
$$  \begin{small} \begin{matrix}
p_{111} = \alpha \beta \gamma ,\qquad & p_{112} = \alpha \beta (1-\gamma),
 \qquad &  p_{121} = \alpha(1-\beta)\gamma , \qquad &  p_{122}= \alpha (1-\beta) (1-\gamma),  \qquad \\
p_{211} = (1-\alpha)\beta \gamma , &\!\! p_{212} = (1-\alpha) \beta(1-\gamma) ,& \! \! 
p_{221} = (1-\alpha)(1-\beta) \gamma ,& \!\! p_{222} = (1-\alpha)(1-\beta)(1-\gamma).
\end{matrix} \end{small}
$$
After this substitution, and after removing extraneous factors, the three $2 \times 2$ determinants
are precisely the three bilinear polynomials exhibited in \cite[Corollary 6.3]{CBMS}. These  equations
have two solutions in $\PP(V)$, so there can be two distinct totally mixed Nash equilibria.
\hfill $\diamond$
\end{example}

For any game $X$, the set of dependency equilibria
contains the set of Nash equilibria. The latter is usually finite.
It is instructive to compare these objects for some examples from
  game theory text books. Some of these games are 
not presented in normal form, but in extensive form. It takes practise to
 derive the payoff tensors $X^{(i)}$ from extensive forms.

\begin{example}[Centipede Game] \label{ex:centipede}
This is a famous class of two-person games due to Robert Rosenthal \cite{Ros}. 
They are presented in extensive form, by graphs that looks
like a centipede. We discuss an instance with $d_1=3, d_2 = 2$.  
Our game is presented by the following graph:
\begin{center}
    \begin{tikzpicture}[font=\footnotesize,scale=1]
\tikzstyle{solid node}=[circle,draw,inner sep=1.2,fill=black];
\tikzstyle{hollow node}=[circle,draw,inner sep=1.2];
\node(0)[hollow node]{}
child[grow=down]{node[solid node]{}edge from parent node[left]{$d$}}
child[grow=right]{node(1)[solid node]{}
child[grow=down]{node[solid node]{}edge from parent node[left]{$d$}}
child[grow=right]{node(2)[solid node]{}
child[grow=down]{node[solid node]{}edge from parent node[left]{$d$}}
child[grow=right]{node(3)[solid node]{}
edge from parent node[above]{$r$}
}
edge from parent node[above]{$r$}
}
edge from parent node[above]{$r$}
};
Page 23 of 28
\foreach \x in {0,2}
\node[above]at(\x){1};
\foreach \x in {1}
\node[above]at(\x){2};
\node[below]at(0-1){$(1,0)$};
\node[below]at(1-1){$(0,2)$};
\node[below]at(2-1){$(3,1)$};
\node[below]at(3){$(2,4)$};
\end{tikzpicture}
\end{center}
The two players chose sequentially between going right $r$ or down $d$.
A down choice ends the game. In our instance, the game also ends after three
right choices.
The payoffs for the four outcomes $d$, $rd$, $rrd$ or $rrrd$
are the labels of the leaves. This translates into a $3 \times 2$-game:
\begin{table}[H]
\centering
    \setlength{\extrarowheight}{2pt}
\begin{tabular}{cc|c|c|c|}
  & \multicolumn{1}{c}{} & \multicolumn{3}{c}{Player $2$} \\
  & \multicolumn{1}{c}{} & \multicolumn{1}{c}{$d$}  & \multicolumn{1}{c}{$r$}  \\\cline{3-4}
            & $d$ & $(1,0)$ & $(1,0)$ \\ \cline{3-4}
 Player $1$ 
 & $r+d$ & $(0,2)$ & $(3,1)$  \\\cline{3-4}
            & $r+r$ & $(0,2)$ & $(2,4)$  \\\cline{3-4}
\end{tabular}
  \end{table} 

This table gives the  $3 \times 2$ payoff matrices $X^{(1)}$ and $X^{(2)}$. 
Similarly to the Prisoner's Dilemma, the Nash equilibrium of the centipede game is not Pareto efficient. To compute the dependency equilibria, we consider four quadrics in six unknowns, namely the $2 \times 2$~minors of the matrices
$M_1$ and $M_2$. The ideal they generate is the intersection of two prime ideals:
$$ \!\begin{matrix}
 \langle \,
 p_{31}- p_{32}\,,\,\,
 p_{21}-2  p_{22}\,,\,\,
p_{11} p_{22}-4 p_{12} p_{22}-2 p_{22}^2+4 p_{11} p_{32}-2 p_{12}p_{32}+3 p_{22} p_{32}+2 p_{32}^2
\, \rangle \qquad \cap  \\
\! \langle\,
p_{11} + p_{12}\,,\,\,
3 p_ {22}  p_{31} - 2  p_{21}  p_{32} +  p_{22}  p_{32}\,, \\
6 p_{12} p_{21} + 3 p_{12} p_{22}  +  3 p_{21} p_{22} + 6 p_{12}p_{31}
+ 12 p_{12} p_{32} - 4 p_{21} p_{32} - p_{22} p_{32}  - 6 p_{31} p_{32}\, \rangle.
 \end{matrix} 
$$
The second  component,
a singular quartic surface in a hyperplane in $\PP^5$, is disjoint from $\Delta$.
The first component is a hyperboloid in a $3$-space $\PP^3$ which intersects
the open simplex $\Delta$. 
That intersection is the set of dependency equilibria. There are no Nash equilbria in $\Delta$.
\hfill $\diamond$
\end{example}

\section{The Spohn variety}
\label{sec3}

In this section we work in the complex projective space $\PP(V)$ 
of $d_1 \times \cdots \times d_n$ tensors.
We write $\mathcal{V}_X$ for the subvariety of $\PP(V)$ that is
given by requiring $M_1,\ldots,M_n$ to have rank one.
We call $\mathcal{V}_X$ the {\em Spohn variety} of the game $X$. 
Thus $\mathcal{V}_X$ is defined by 
$\sum_{i=1}^n\binom{d_i}{2} $ quadratic forms in $\prod_{i=1}^n d_i$ unknowns
$p_{j_1\cdots j_n}$, namely the
$2 \times 2$ minors of the $n$ matrices $M_i$ in (\ref{eq: Spohn matrices}).

We already saw several examples in the previous section.
For three-person games with binary choices   (Example \ref{ex:twotwotwo}),
the Spohn variety
$\mathcal{V}_X$ is a fourfold in $\PP^7$.
For the centipede game (Example \ref{ex:centipede}), the Spohn variety
$\mathcal{V}_X$ is a surface in $\PP^8$. We next consider a $2 \times 2$ game.

 \begin{example}[Bach or Stravinsky]
For the game in Example \ref{ex:bach}, we consider the matrices
$$ M_1 \,=\, \begin{bmatrix}
\,p_{11} + p_{12} &\, 3 p_{11} \\
\,p_{21} + p_{22}  & \,2 p_{22} 
\end{bmatrix} \quad {\rm and}  \quad M_2 \,=\,  \begin{bmatrix}
\,p_{11} + p_{21} & \, 2 p_{11} \\
\, p_{12} + p_{22} & \, 3 p_{22} 
\end{bmatrix}.
$$
The ideal generated by ${\rm det}(M_1)$ and ${\rm det}(M_2)$ 
is the intersection of three prime ideals:
$$
\langle p_{11},p_{22} \rangle \, \cap \,
\langle 2 p_{12} + 3p_{21},  3p_{11}p_{21} + p_{11}p_{22} + 3p_{21}p_{22}  \rangle \,\cap\,
\langle 2p_{12} - 3p_{21} - p_{22}, p_{11} - p_{22} \rangle. $$
This shows that the Spohn variety $\mathcal{V}_X$ is
the reduced union of three curves, two lines and~one conic,
shown in Figure \ref{fig: dependency equilibria}.
Only one component, namely a line, intersects the open tetrahedron $\Delta$.
This game has two pure Nash equlibria $(1,0,0,0), (0,0,0,1)$ 
and one totally mixed Nash equilibrium $(\frac{6}{25},\frac{9}{25},\frac{4}{25},\frac{6}{25})$.
The latter is the positive point of rank one on the curve $\mathcal{V}_X$.
\hfill $\diamond$ \end{example}

\begin{figure}[H]
\centering \vspace{-0.3cm}
\includegraphics[width=8.8cm]{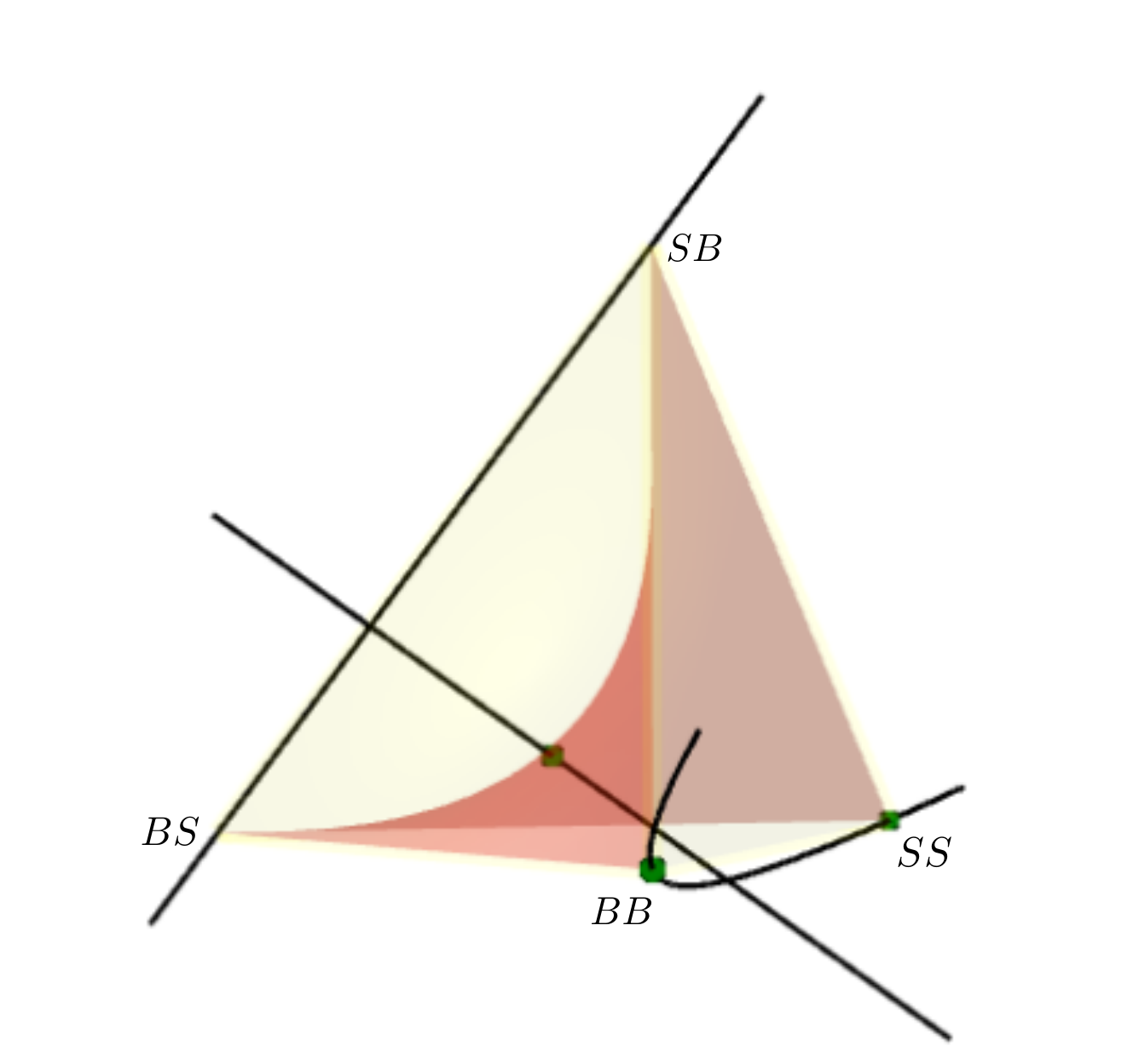}
\vspace{-0.2cm}
\caption{The Spohn variety is a reducible curve of degree four in $\PP^3$.
It has three components but only one passes through the tetrahedron.
The figure also shows the Segre surface in the tetrahedron. 
The curve and the surface meet in one point, namely
the totally mixed Nash equilibrium.
    \label{fig: dependency equilibria}}
\end{figure} 

The curve in Figure \ref{fig: dependency equilibria} has multiple components
because the payoff matrices in (\ref{eq:bach}) are very special.
If we perturb the matrix entries, then 
 the resulting curve $\mathcal{V}_X$ will be smooth and irreducible in $\PP^3$.
 As we shall see, the analogous result holds for games of arbitrary size.
 
 We now present our first result in this section.
It summarizes the essential geometric features of Spohn varieties, and it shows
how these varieties are related to  the Nash equilibria.

\begin{theorem} \label{thm:spohn}
If the payoff tables $X$ are generic then the Spohn variety $\mathcal{V}_X$ 
is irreducible of codimension $d_1 + d_2 + \cdots + d_n - n$ and degree 
$d_1 d_2 \ldots d_n$. The intersection of
$\mathcal{V}_X$ with the Segre variety in the  open simplex $\Delta$ is precisely
the set of totally mixed Nash equilibria for~$X$.
\end{theorem}

\begin{proof}
Consider a generalized column of 
the $d_i \times 2$ matrix $M_i$, i.e.\ a  linear combination of  the columns of $M_i$ with coefficients $\lambda_1, \lambda_2 \in \RR$ that are not both zero.  
Since the payoff table $X^{(i)}$ is generic, every generalized column of $M_i$ consists of linearly independent linear forms. 
We know from \cite[Theorem 6.4]{Eisenbud2004TheGO} that the ideal generated by 
the $2 \times 2$ minors of $M_i$ is prime of codimension $d_i-1$. Moreover, by \cite[Proposition 2.15]{brunsvetter},
the degree of this linear determinantal variety is $\binom{2 + d_i-1-1}{d_i -1} =d_i$. Since 
the tensor $X^{(i)}$ is generic and its entries occur only in $M_i$, the intersection of the $n$ varieties is transversal. 
Now, \cite[Theorem 1.24]{shafarevich} and B\'ezout's Theorem for dimensionally transverse intersections yield the first assertion. 

The second assertion says that the totally mixed Nash equilibria are the dependency~equilibria of rank one.
Set $p_{j_1 \ldots j_n}=\pi^{(1)}_{j_1} \cdots \pi^{(n)}_{j_n}$
with $\pi^{(i)}_k >0$ for $k \in [d_i]$ and $i \in [n]$. Suppose (\ref{eq:sumtoone}) holds. The dependency equilibria of rank one
are defined by the $2 \times 2$ minors of the matrix
\[\begin{bmatrix}
\,\,\,1 \quad & \sum_{j_1 = 1}^{d_1} \cdots \widehat{\sum_{j_{i} = 1}^{d_{i}}} \cdots \sum_{j_n = 1}^{d_n} X^{(i)}_{j_1 \cdots 1 \cdots j_n} \pi^{(1)}_{j_1} \cdots \pi^{(i-1)}_{j_{i-1}}\pi^{(i+1)}_{j_{i+1}}\cdots \pi^{(n)}_{j_{n}}   
\smallskip \\
\,\,\,\vdots \quad & \vdots \\
\,\,\,1 \quad & \sum_{j_1 = 1}^{d_1} \cdots \widehat{\sum_{j_{i} = 1}^{d_{i}}} \cdots \sum_{j_n = 1}^{d_n} X^{(i)}_{j_1 \cdots d_i \cdots j_n} \pi^{(1)}_{j_1} \cdots \pi^{(i-1)}_{j_{i-1}}\pi^{(i+1)}_{j_{i+1}}\cdots \pi^{(n)}_{j_{n}} 
\end{bmatrix}\!.\]
We subtract the first row from  the $k$th row for all $k \in \{2, \ldots, d_i\}$. 
The $2 \times 2$ minors of the resulting matrix are
the  pairwise differences of the entries in the second column.
These differences are precisely
the $d_i -1$ multilinear equations exhibited in \cite[Theorem 6.6]{CBMS}.
\end{proof}

The Spohn variety $\mathcal{V}_X$ is a high-dimensional projective
variety associated with a game~$X$. Each point $P$ on $\mathcal{V}_X$ is
a tensor. We say that $P$ is a {\em Nash point} if that tensor has rank one.
The positive Nash points in $\mathcal{V}_X \cap \Delta$ are the
totally mixed Nash equilibria. Their number is given by the 
formula in \cite[Section 6.4]{CBMS}, namely it  expressed as the mixed volume of 
certain products of simplices. That mixed volume is zero when the
tensor format is too unbalanced.

\begin{remark}
A generic game $X$ has no Nash points unless
\begin{equation}
\label{eq:balanced} 
\quad d_i  \,\,\leq \, \,d_1 + \cdots + d_{i-1} + d_{i+1} + \cdots + d_n -n+2 \qquad
\hbox{for}\,\,\,\, i=1,2,\ldots,n.
\end{equation}
Experts on tensor geometry recognize these inequalities from
a result by Gel'fand, Kapranov and Zelevinsky
on hyperdeterminants~\cite[Theorem 14.I.1.3]{GKZ}.
Namely, the existence of Nash points for a given tensor format is equivalent to 
the hyperdeterminant being a hypersurface.
In particular, two-person games have Nash points if and only if the matrix is square 
($d_1=d_2$).
\end{remark}

We continue to assume that the payoff tables are generic. Then the following result holds.

\begin{theorem} \label{thm:rational}
If $n= d_1 = d_2 = 2$ then the Spohn variety 
$\mathcal{V}_X$ is  an elliptic curve. In all other cases, the Spohn variety
$\mathcal{V}_X$ is rational, represented by a map onto  $(\PP^1)^n$ with linear fibers.
\end{theorem}

\begin{proof}
We shall provide a 
parametrization of $\mathcal{V}_X$. Along the way, we shall see why
the case $n=d_1=d_2=2$ is special.
The entries of these $n$ matrices $M_i$ in (\ref{eq: Spohn matrices})
are linear forms in the entries $p_{j_1\cdots j_n}$ of
the tensor $P$. Their coefficients
depend linearly on the entries of~$X$.

Consider the affine line whose coordinate $x_i = P X^{(i)}$
is the expected payoff (\ref{eq:expected}) for player~$i$.
We embed this into a projective line $\PP^1$ by setting
 $z_i = (x_i:-1)$. We call  $(\PP^1)^n$ 
the {\em algebraic payoff space}. Its homogeneous coordinates are
  $z = (z_1,z_2,\ldots,z_n)$.
The {\em algebraic payoff map} is the following rational map from 
the Spohn variety to the algebraic payoff space:
\begin{equation}
\label{eq:algpayoff} \pi_X : \mathcal{V}_X\, \dashrightarrow \,(\PP^1)^n \,,\,\,
P\, \mapsto\,  \bigl(\,
{\rm ker}(M_1(P))\,, \,
{\rm ker}(M_2(P))\,, \,\ldots, \,{\rm ker}(M_n(P)) \,\bigr). 
\end{equation}

The name ``payoff map'' is justified as follows.
Suppose that $P$ is a dependency equilibrium, so $P$ is a point in the set
$\mathcal{V}_X  \cap \Delta$. 
The expected payoff $x_i$ for the $i$th player satisfies 
\begin{equation}
\label{eq:Mx}
M_i(P) \cdot \begin{bmatrix} \phantom{-} x_i\, \\ -1 \,\end{bmatrix} \,= \, 0 \quad {\rm for} \quad
i=1,2,\ldots,n.
\end{equation}
To see this, augment the rank one matrix $M_i(P)$ by its row of column sums, like in
(\ref{eq:whoserankisone}).
Equation (\ref{eq:Mx}) implies 
$ \pi_X(P)  = \bigl( (x_1\!:\!-1),\ldots,(x_n\!:\!-1) \bigr)$. We now write (\ref{eq:Mx}) on
$(\PP^1)^n$ as~follows:
\begin{equation}
\label{eq:MzUP}
 M_i(P) \cdot z_i^T \, = \, K_{X,i}(z_i) \cdot  P , 
 \end{equation}
where the tensor $P$ is vectorized as column.
The matrix $K_{X,i}(z_i)$ has $d_i$ rows and $d_1d_2\cdots d_n$ columns. Its
entries are binary forms in $z_i$ whose coefficients depend on the entries of $X^{(i)}$.

\begin{definition}
The {\em Konstanz matrix} $K_X(z)$ of the game $X$ is a matrix with
$\sum_{i=1}^n d_i$  rows and $d_1 d_2 \cdots d_n$ columns.
It is obtained by stacking the matrices $K_{X,1}(z_1), \ldots, K_{X,n}(z_n)$
on top of each other.
When working on the affine chart $z_i = (x_i:-1)$, we write
$K_X(x)$.
\end{definition}

The Konstanz matrix $K_X(z)$ has linearly independent  rows
when $z$ is generic.
Therefore, its kernel is a vector space of dimension
 $D = \prod_{i=1}^n d_i - \sum_{i=1}^n d_i$.
 We regard  ${\rm ker}(K_X(z))$ as a
linear  subspace of dimension $D-1$ in the projective space $\PP(V)$.
 Our construction implies that the Spohn variety 
 is the union of these linear spaces for $z \in (\PP^1)^n$:
\begin{equation} 
\label{eq:spohnpara}
\mathcal{V}_X \quad =  \quad \bigl\{ \,P \in \PP(V) \,:\,
K_X(z) \cdot P = 0 \, \,\,\hbox{for some} \,z \in (\PP^1)^n \,\bigr\}. 
\end{equation} 

At this point we must distinguish the cases $D \geq 1$ and $D = 0$. First, let
$D\geq 1$. Then the map  $\pi_X$ is dominant, and its generic fiber
is a linear space $\PP^{D-1} $. This map furnishes
an explicit birational isomorphism
between the Spohn variety  $\mathcal{V}_X$ and $\PP^{D-1} \times (\PP^1)^n$.
The representation (\ref{eq:spohnpara}) gives the inverse, hence the
desired rational parametrization of  
$\mathcal{V}_X$. This confirms the dimension formula 
in Theorem \ref{thm:spohn}, which is here rewritten as
${\rm dim}(\mathcal{V}_X) = D-1+n$.

Finally, let $D=0$.
This implies $n=d_1=d_2 = 2$, so the Konstanz matrix has format 
$4 \times 4$. It is shown in  (\ref{eq:Konstanz22}).
The determinant of $K_X(z)$ is a curve of degree $(2,2)$
in $\PP^1 \times \PP^1$, so it is an elliptic curve.
The map $\pi_X$ gives a birational isomorphism
from $\mathcal{V}_X$ onto this curve.
This elliptic curve is studied in detail in Section \ref{sec4},
and we will revisit it in Example \ref{ex:typically}.
\end{proof}

The case $D=1$ is also of special interest, because here
$\pi_X$ is a birational isomorphism.

\begin{example}[Del Pezzo surfaces of degree two] \label{ex:delpezzo}
Let $n=2, d_1=3,d_2 = 2$. Up to relabelling,
this is the only case satisfying $D=1$.
The Konstanz matrix equals
\begin{equation}
\label{eq:Konstanz32} K_X(x) \quad = \quad
\begin{bmatrix}
\,x_1-a_{11} & x_1-a_{12} & 0 & 0 & 0 & 0 \\
0 & 0 & x_1-a_{21} & x_1-a_{22} & 0 & 0  \\
0 & 0 & 0 & 0 & x_1-a_{31} & x_1-a_{32} \,\\
\, x_2 - b_{11} & 0 & x_2 - b_{21} & 0 & x_2 - b_{31} & 0 \\
0 & x_2 - b_{12} & 0 & x_2 - b_{22} & 0 & x_2 - b_{32} \,
\end{bmatrix}.
\end{equation}
Here $(x_1,x_2)$ are coordinates on an affine chart $\CC^2$ of $\PP^1 \times \PP^1$.
The rank of (\ref{eq:Konstanz32}) drops from $5$ to $4$ at precisely six
points in $\PP^1 \times \PP^1$. Five of these lie in $\CC^2$.  We obtain a rational map
$$ 
\PP^1 \times \PP^1 \dashrightarrow \PP^5 \, , \,
(x_1,x_2) \mapsto {\rm ker}(K_X(x)).
$$
This blows up six points, and its image is the Spohn surface $\mathcal{V}_X$. 
The inverse map is  $\pi_X$.
We conclude that $ \mathcal{V}_X$ is the blow-up of $\PP^1 \times \PP^1$
at six general points. When seen through the lens of
algebraic geometry \cite[Example~1.9]{RanRoc},
this is a del Pezzo surface of degree two.
\hfill $\diamond$
\end{example}

Konstanz matrices for three other tensor formats are shown in
Examples  \ref{ex:typically},
\ref{ex:33Konstanz} and~\ref{ex:222Konstanz}.

\section{Elliptic Curves}
\label{sec4}

In this section we take a closer look at $2 \times 2 $  games, with
 payoff matrices 
 $X^{(1)} = (a_{ij})$ and  $X^{(2)} = (b_{ij})$. The Spohn variety $\mathcal{V}_X$
is the elliptic curve in $\PP^3$ defined by the two quadrics
$$ \begin{matrix}
\! f_1 \,=\, {\rm det}(M_1)\,=
       (a_{21} - a_{11}) p_{11} p_{21} + (a_{22} - a_{11}) p_{11} p_{22}
    + (a_{21} - a_{12}) p_{12} p_{21} + (a_{22} - a_{12}) p_{12} p_{22},\,
\\
 f_2 \,\,=\,{\rm det}(M_2) \, = 
 (b_{12} - b_{11}) p_{11} p_{12} \,+ (b_{22} - b_{11}) p_{11} p_{22} 
 \,+\, (b_{12} - b_{21}) p_{12} p_{21} +\, (b_{22} - b_{21}) p_{21} p_{22}.\,
\end{matrix}
$$
This curve passes through the coordinate points
$E_{11},E_{12},E_{21},E_{22}$ in $\PP^3$. It is smooth
and irreducible when $a_{ij}$ and $b_{ij}$ are generic.
A  planar model of this elliptic curve is obtained by eliminating $p_{22}$
from $f_1$ and $f_2$. Setting $\,p_{11} = x,\, p_{12} = y,\, p_{21} = z$, 
we find the ternary~cubic 
\begin{equation}
\label{eq:eliminant}  \begin{matrix}
\quad (a_{11}-a_{22})(b_{11}-b_{12}) \,x^2y \,+\, (a_{11}-a_{21})(b_{22}-b_{11})\,x^2z
     \,+\,(a_{12}-a_{22})(b_{11}-b_{12}) \,x y^2 \\ +\, (a_{11}-a_{21}) (b_{22}-b_{21}) \,x z^2
     +\,(a_{12}-a_{22}) (b_{21}-b_{12}) \,y^2 z \,+\, (a_{12}-a_{21}) (b_{22}-b_{21})\, y z^2  \\
   \qquad \qquad  \qquad \qquad \qquad\qquad
   +\,\bigl(\,(a_{12}-a_{21}) (b_{22}-b_{11}) \,+\, (a_{11}-a_{22})(b_{21}-b_{12})\,\bigr)  \,xyz.
     \end{matrix}
\end{equation}

A ternary cubic of the form (\ref{eq:eliminant}) is called a {\em Spohn cubic}.
This passes through the
three coordinate points in $\PP^2$. But there are other restrictions. To see this,
we consider all cubics
\begin{equation}
\label{eq:cubicofform}
c_1 x^2 y\,+\,c_2 x^2 z\,+\,c_3 x y^2\,+\,c_4 x z^2\,+\,c_5 y^2 z\,+\,c_6 y z^2\,+\,c_7 x y z.
\end{equation}
The set of such cubics is a projective space $\PP^6$ with homogeneous coordinates
$c_1,\ldots,c_7$.

\begin{proposition}
The Spohn cubics (\ref{eq:eliminant}) form the $4$-dimensional variety in $\PP^6$ 
given by
$ c_1+c_2-c_3-c_4+c_5+c_6-c_7 = 
c_2 c_4 c_5-c_3 c_4 c_5-c_2 c_3 c_6+c_4 c_5 c_6+c_3 c_4 c_7-c_4 c_5 c_7-c_4^2 c_5+c_4 c_5^2
= 0$.
This is a cubic hypersurface inside a hyperplane $\PP^5$. Its singular locus consists of
 nine points.
\end{proposition}

\begin{proof}
This is obtained by a direct computation using the software {\tt Macaulay2} \cite{M2}.
\end{proof}
 
 While the general Spohn cubic is smooth, it can be singular
 for special payoff matrices. To identify these,  we compute
the discriminant $\mathcal{D}$ of the ternary cubic (\ref{eq:cubicofform}).
This discriminant is an irreducible polynomial of degree $12$ 
in seven unknowns. It is a sum of $127$ terms:
$$ \mathcal{D} \,\,=\,\, 16  c_1^5  c_4^2  c_5^2  c_6^3+16  c_1^4  c_2^2  c_5^2  c_6^4-24  c_1^4  c_2  c_4^2  c_5^3  c_6^2 +\, \cdots\, +c_2^2  c_3^2  c_4^2  c_5^2  c_7^4-c_2^2  c_3^2  c_4  c_5  c_6  c_7^5.
$$
We now plug in the Spohn cubic (\ref{eq:eliminant}).
The resulting discriminant
is a polynomial of degree $24$ in the
eight unknowns $a_{ij},b_{ij}$. It factors into nine irreducible factors, namely
$\, \mathcal{D}(a,b) =$
$$  (a_{11}-a_{12})^2 (a_{11}-a_{21})^2	(a_{12}-a_{22})^2 (a_{21}-a_{22})^2	
(b_{11}-b_{12})^2 (b_{11}-b_{21})^2 (b_{12}-b_{22})^2 (b_{21}-b_{22})^2	
\mathcal{E}(a,b).
$$
The last factor  $\mathcal{E}(a,b)$  has $587$ terms of
 degree $8$.
Nonvanishing of the discriminant
$\mathcal{D}(a,b)$ ensures that the Spohn cubic
(\ref{eq:eliminant}) is smooth in $\PP^2$, and hence so is the
curve $\mathcal{V}_X$ in $\PP^3$.

We have argued that the general Spohn curve $\mathcal{V}_X$ is an elliptic curve.
It is thus natural to express its {\em j-invariant}, which identifies
the isomorphism type, in terms of the payoff matrices.

\begin{proposition} \label{prop:jinv}
The j-invariant of the Spohn cubic equals
$\, \mathcal{I}(a,b)^3/\mathcal{D}(a,b) $,
where $\mathcal{I}(a,b)$ is an irreducible polynomial
of degree $8$ with $633$ terms in the entries of the two payoff tables.
\end{proposition}

\begin{proof} For any ternary cubic, the j-invariant is the cube of the Aronhold invariant 
divided by the discriminant; see \cite[Example 11.12]{MS}. Here,
$\mathcal{I}(a,b)$ is the Aronhold invariant of (\ref{eq:eliminant}).
\end{proof}

The dependency equilibria of our game are the points in
  $\mathcal{V}_X \cap \Delta$.
 To better understand this semialgebraic set, we identify
 some landmarks on the curve $\mathcal{V}_X$.
 The first such landmark is the Nash point, which is the unique  rank one 
 matrix in $\PP^3$ lying on $\mathcal{V}_X$:
  \begin{equation}
\label{eq:nash22}
 N \,\,=\,\, \begin{bmatrix} b_{22} - b_{21} \\ b_{11} - b_{12} \end{bmatrix}
\begin{bmatrix} a_{22}-a_{12} & a_{11} - a_{21} \end{bmatrix}
\end{equation}
Suppose that the following holds and the two signs are non-zero:
\begin{equation}
\label{eq:signcondition}
 {\rm sign}(a_{11}-a_{21})= {\rm sign}(a_{22}-a_{12}) 
\quad {\rm and} \quad {\rm sign}(b_{11}-b_{12})={\rm sign}(b_{22}-b_{21}). 
\end{equation}
Then we can scale the matrix $N$ in (\ref{eq:nash22}) by
$\bigl(( a_{11} - a_{21}+a_{22}-a_{12} ) (b_{11} - b_{12}+b_{22} - b_{21} )\bigr)^{-1}$
to land in $\Delta$, and the result is the
 unique  totally mixed Nash equilibrium of the game.
 
Next recall that the four coordinate points $E_{ij}$
lie on the curve $\mathcal{V}_X$.
Their tangent lines
${\rm span}(D_{ij},E_{ij})$ are specified by their intersection points
with the opposite coordinate planes:
$$ \begin{small}
\begin{matrix}
\! D_{11} = \begin{bmatrix}  0 \!\! & \!\! (a_{11}{-}a_{21}) (b_{22}{-}b_{11}) \\ 
 (a_{22}{-}a_{11}) (b_{11}{-}b_{12}) \!\! & \!\! (a_{11}{-}a_{21})(b_{11}{-}b_{12} )\end{bmatrix} \!,
D_{12} = \begin{bmatrix} (a_{22}{-}a_{12})(b_{12}{-}b_{21}) \!\! & \!\! 0 \\ 
 (a_{22}{-}a_{12})(b_{11}{-}b_{12}) \! \! &\!\! (a_{12}{-}a_{21})(b_{11}{-}b_{12}) \end{bmatrix}  \medskip \\
\! D_{21} =  \begin{bmatrix} (a_{21}{-}a_{12})(b_{21}{-}b_{22}) \!\! & \!\!
   (a_{11}{-}a_{21})(b_{21}{-}b_{22}) \\
    0 \!\! & \!\! (a_{11}{-}a_{21})(b_{12}{-}b_{21})\end{bmatrix}\!,
D_{22} = \begin{bmatrix}(a_{22}{-}a_{12})(b_{21}{-}b_{22}) \!\! & \!\!
 (a_{11}{-}a_{22})(b_{21}{-}b_{22}) \\ (a_{12}{-}a_{22})(b_{11}{-}b_{22} )
 \!\! & \!\! 0 \end{bmatrix}
\end{matrix}  \end{small}
$$
And, finally, our curve intersects each coordinate plane in a unique non-coordinate point:
$$ \begin{small}
\begin{matrix}
\! F_{11} = \begin{bmatrix} 0 \!\! & \!\! 
(a_{12}{-}a_{21})(b_{21}{-}b_{22}) \\
(a_{12}{-}a_{22})(b_{21}{-}b_{12}) \!\! & \!\! (a_{12}{-}a_{21})(b_{12}{-}b_{21}) 
\end{bmatrix} \! , 
F_{12} =  \begin{bmatrix}(a_{11}{-}a_{22})(b_{21}{-}b_{22}) \!\! & \!\! 0 \\
 (a_{11}{-}a_{22})(b_{22}{-}b_{11}) \!\! & \!\! (a_{11}{-}a_{21})(b_{11}{-}b_{22})
 \end{bmatrix} \medskip  \\
 \! F_{21} = \begin{bmatrix} 
 (a_{12}{-}a_{22})(b_{11}{-}b_{22}) \!\! & \!\! 
 (a_{11}{-}a_{22})(b_{22}{-}b_{11}) \\
  0 \!\! & \!\! (a_{11}{-}a_{22})(b_{11}{-}b_{12}) \end{bmatrix} \! ,
 F_{22} = \begin{bmatrix} (a_{12}{-}a_{21})(b_{12}{-}b_{21}) \!\! & \!\!
  (a_{11}{-}a_{21})(b_{21}{-}b_{12}) \\
   (a_{12}{-}a_{21})(b_{11}{-}b_{12}) \!\! & \!\! 0 \end{bmatrix}
   \end{matrix}
   \end{small}
$$
We now show that dependency equilibria may exist even if there are no
Nash equilibria in~$\Delta$:

\begin{example}[Disconnected equilibria]  \label{ex:disconnected}
Consider the game $X$ given by the payoff matrices
$$ 
\begin{bmatrix}
a_{11} \! &\! a_{12} \\ a_{21} \! & \! a_{22} \end{bmatrix}
 =  \begin{bmatrix}
 2 & 0 \\ 4 & 1
\end{bmatrix} \quad {\rm and} \quad
\begin{bmatrix}
b_{11} \! &\! b_{12} \\ b_{21} \! & \! b_{22} \end{bmatrix}
 =  \begin{bmatrix}
 2 & 1 \\ 4 & 3 \end{bmatrix}, \quad \hbox{with Nash point} \,\,
 N \, = \, \begin{bmatrix} -1 & 2 \\ \phantom{-}1 &\!\!\! -2 \end{bmatrix}\! .
$$
Here, $\mathcal{V}_X$ is smooth and irreducible. This
elliptic curve has j-invariant
$- (7^3 103^3)/(2^8 3^2 47)$.
The real curve $\mathcal{V}_X \cap \Delta$ 
has two connected components, both disjoint from
the Segre surface $\langle p_{11} p_{22} -p_{12} p_{21}\rangle$.
One arc connects $E_{11}$ and $F_{21}$, and the other
arc connects $E_{22}$ and $F_{12}$.
\hfill $\diamond$
\end{example}

The combinatorics of the curve    $\mathcal{V}_X \cap \Delta$
 is given by the signs of the entries in the nine matrices
  $N$, $D_{ij}$ and $F_{ij}$. These signs
  are determined by the respective orderings
 of $a_{11},a_{12},a_{21},a_{22}$ and
 $b_{11},b_{12},b_{21},b_{22}$, assuming that these
 are quadruples of distinct numbers.
 We derive the following theorem
 by analyzing all   $(4!)^2 = 576$ possibilities for these pairs of
 orderings.

\begin{theorem} \label{thm:arcs}
For a generic $2 \times 2$ game $X$,
the curve of dependency equilibria $\mathcal{V}_X \cap \Delta$ 
has either $\,0$, $1$ or $2$ connected components,
each of which is an arc between two boundary points.
If (\ref{eq:signcondition}) holds then there is exactly one $EE$, $EF$ or $FF$ arc.
If (\ref{eq:signcondition}) does not hold then all components are
$EF$ arcs, and their number can be $\,0$, $1$ or $2$.
\end{theorem}

\section{The Payoff Region}
\label{sec5}

The $n$ payoff tensors $X^{(i)}$ define a canonical linear map from tensor space 
to payoff space:
\begin{equation}
\label{eq:payoffmap}
 \pi_X \,: \, V \rightarrow \RR^n\,,\,\, \, P \,\,\mapsto\, \,\bigl(P X^{(1)} , \,PX^{(2)} , \,\ldots\,, \,P X^{(n)} \bigr). 
 \end{equation}
The $i$th coordinate  $P X^{(i)}$ is the
expected payoff for player $i$, given by the formula in
(\ref{eq:expected}).
We call $\pi_X$ the {\em payoff map}.
By (\ref{eq:Mx}), this is the lifting to $V$ of the
algebraic payoff map in~(\ref{eq:algpayoff}).

The image of the probability simplex
$\Delta$ is a convex polytope $\pi_X(\Delta)$  that is usually full-dimensional in $\RR^n$.
This polytope is known as the {\em cooperative payoff region} of the game $X$.
Its points are all possible expected payoff vectors for the game in question.
Tu and Jiang \cite{TJ} investigate the semialgebraic subset that is obtained
by projecting all rank one tensors in $\Delta$. This is a nonconvex subset
of $\pi_X(\Delta)$, known as the {\em noncooperative payoff region}. 

For $2 \times 2$ games, this region is the image of the Segre surface
under a linear projection into the plane. Our readers might like to
compare \cite[Figure 1]{TJ} with the surface
shown in Figure \ref{fig: dependency equilibria}.

We are interested in the subset of payoff vectors that arise from dependency equilibria:
$$ \mathcal{P}_X \,:= \,\pi_X( \mathcal{V}_X \cap \Delta )  \,\subset \, \pi_X(\Delta) \,\subset \, \RR^n . $$
The set $\mathcal{P}_X$ is semialgebraic, by Tarski's Theorem on Quantifier Elimination.
The authors of \cite{TJ} would probably call $\mathcal{P}_X$
the {\em dependency payoff region} of the game $X$. In 
the present paper, we just use the term {\em payoff region} for $\mathcal{P}_X$,
since our focus is on dependency equilibria.

We begin by noting that, at every dependency equilibrium of $X$, the expected payoffs agree with the
various conditional expected payoffs. We can thus use conditional expectations in
(\ref{eq:payoffmap})
to define the payoff region $\mathcal{P}_X$.
This is the content of the following lemma.

\begin{lemma} Let $P$ be a tensor in $V$ with $p_{++\cdots+} = 1$ that represents a point in
     $\mathcal{V}_X$. Then 
     \begin{equation}
     \label{eq:cep}
PX^{(i)} \,\,=\,\, \sum_{j_1 = 1}^{d_1} \cdots \widehat{\sum_{j_{i} = 1}^{d_{i}}} \cdots \sum_{j_n = 1}^{d_n} X^{(i)}_{j_1 \cdots   k  \cdots j_n} \frac{p_{j_1 \cdots  k  \cdots j_n}}{p_{+\cdots+k+\cdots+}}
\quad \hbox{for all  $\,i \in [n]\,$ and $ \,k \in [d_i]$.}
\end{equation}
\end{lemma}

\begin{proof}
The $d_i \times 2$ matrix $M_i$ in (\ref{eq: Spohn matrices}) has rank one, by
definition of $\mathcal{V}_X$.
We replace the first row by the sum of all rows. This transforms $M_i$  into the following matrix
whose rank is one:
\begin{equation}
\label{eq:whoserankisone}
\begin{small}
\begin{bmatrix}
\,1 & PX^{(i)} \\
\,p_{+\cdots + 2 + \cdots+} & \sum_{j_1 = 1}^{d_1} \cdots \widehat{\sum_{j_{i} = 1}^{d_{i}}} \cdots \sum_{j_n = 1}^{d_n} X^{(i)}_{j_1 \cdots  2  \cdots j_n} p_{j_1 \cdots  2  \cdots j_n} \\
\,\vdots & \vdots \\
\,p_{+\cdots + d_i + \cdots+} & \sum_{j_1 = 1}^{d_1} \cdots \widehat{\sum_{j_{i} = 1}^{d_{i}}} \cdots \sum_{j_n = 1}^{d_n} X^{(i)}_{j_1 \cdots  d_i  \cdots j_n} p_{j_1 \cdots  d_i  \cdots j_n}
\end{bmatrix}\! . 
\end{small}
\end{equation}
The $2 \times 2$ minor given by the first row and the $k$th row is zero;
see also (\ref{eq:Mx}).
This implies the desired identity (\ref{eq:cep}) for $k \geq 2$.
The case $k=1$ is obtained by swapping rows in $M_i$.
\end{proof}

\begin{example}[$2 \times 2$ games] \label{ex:typically}
The polygon $\pi_X(\Delta)$ is the convex hull in $\RR^2$ of
the points 
$(a_{11},b_{11})$,
$(a_{12},b_{12})$,
$(a_{21},b_{21})$ and
$(a_{22},b_{22})$, so 
it is typically a triangle or a quadrilateral.
This polygon contains the {\em payoff curve} $\mathcal{P}_X$, which is the
image of the curve $\mathcal{V}_X \cap \Delta$ under the payoff map $\pi_X$.
This is a plane cubic, defined by the determinant of the
Konstanz matrix
\begin{equation}
\label{eq:Konstanz22}
 K_X(x) \quad = \quad
\begin{bmatrix}
 x_1 - a_{11} & x_1-a_{12} & 0 & 0 \\
 0 & 0 & x_1 - a_{21} & x_1-a_{22} \\
 x_2 - b_{11} & 0 & x_2 - b_{21} & 0 \\
 0 & x_2 - b_{12} & 0 & x_2 -b_{22} 
\end{bmatrix}.
\end{equation}
For each point $x $ on this curve, the kernel of (\ref{eq:Konstanz22})
 gives the unique matrix $P$ satisfying
$\pi_X(P) = x$.
The payoff region $\mathcal{P}_X$ is the subset of points $x$ on the curve
for which $P > 0$. 

Figure \ref{fig: payoffcurve1}
shows the payoff region for  the Bach or Stravinsky game in Example~\ref{ex:bach}.
It is the blue arc inside the yellow triangle
$ \pi_X(\Delta) = {\rm conv}\{(0,0),(2,3),(3,2)\}$.
This picture  is the image of  Figure 
\ref{fig: dependency equilibria}
under the payoff map $\pi_X$.
Figure \ref{fig: payoffcurve2} shows a perturbed version, with $a_{11}=3.3$ and $b_{22}=3.2$,
where the Spohn curve is irreducible.
\hfill $\diamond$
\end{example}
\begin{figure}
\centering
\begin{subfigure}{.5\textwidth}
  \centering
  \includegraphics[width=.9\linewidth]{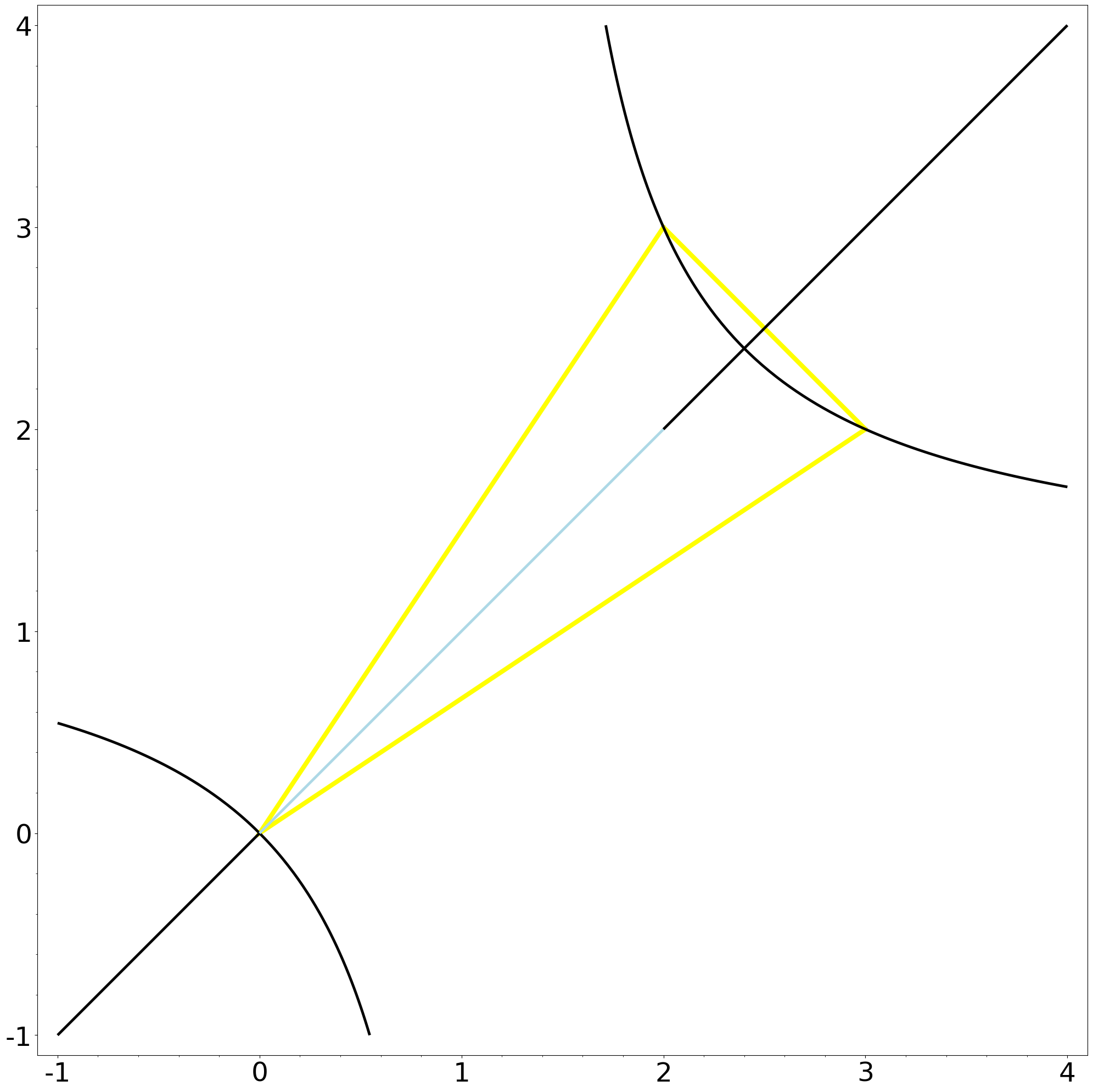}
  \caption{}
\label{fig: payoffcurve1}
\end{subfigure}%
\begin{subfigure}{.5\textwidth}
  \centering
  \includegraphics[width=.9\linewidth]{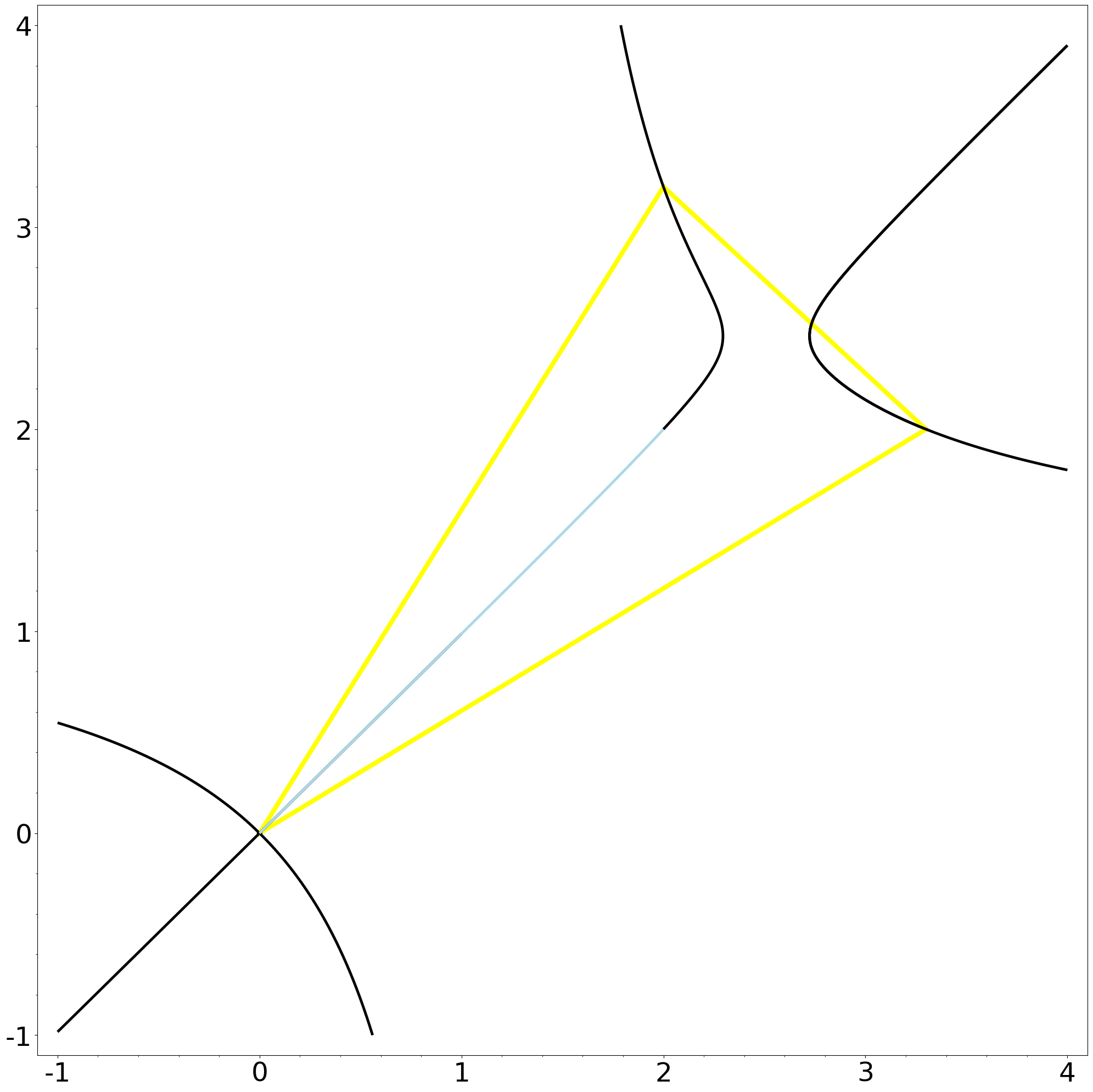}
  \caption{}
\label{fig: payoffcurve2}
\end{subfigure}
\caption{The payoff region for each of these $2 \times 2$ games is the blue arc in the yellow triangle.}
\label{fig: payoff curve}
\end{figure}

We now consider cases other than $2 \times 2$ games, so that
${\rm dim}(\mathcal{V}_X)  \geq n$ holds.
We further assume that $X$ is generic and that $\mathcal{V}_X \cap \Delta$ is non-empty.
Since the algebraic payoff map $\pi_X$  in (\ref{eq:algpayoff}) is dominant,
the payoff region $\mathcal{P}_X$ is a full-dimensional semialgebraic subset of~$\RR^n$.

\begin{figure}[h]
\centering
\includegraphics[width=13cm]{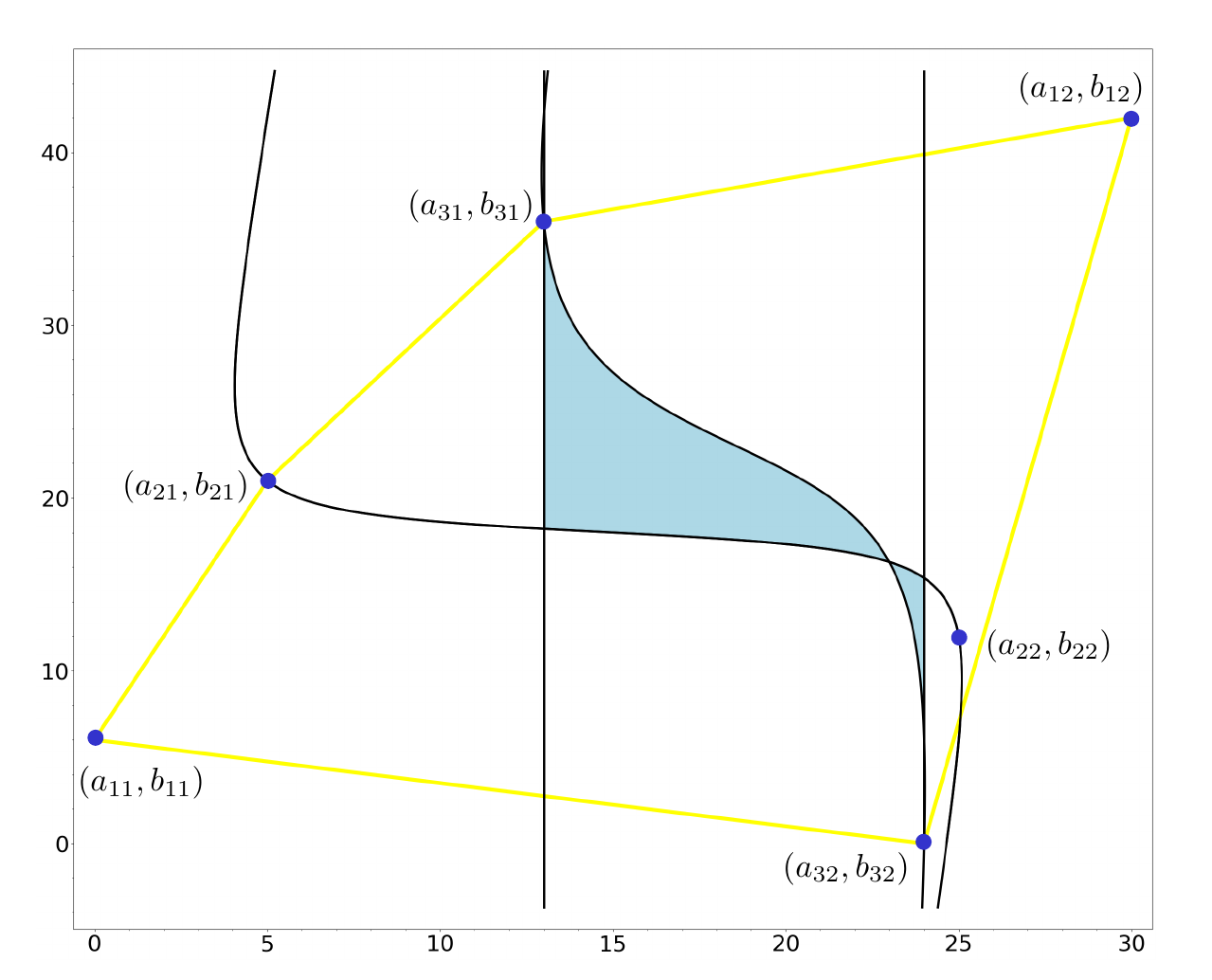}
\caption{The payoff region $\mathcal{P}_X$ for the $3 \times 2$ game in Example \ref{ex:23generic}
consists of two curvy triangles, inside the pentagon $\pi_X(\Delta)$.
Its boundary is given by two lines and two cubics.
    \label{fig: payoff region}}
\end{figure}

\begin{example}[$3\times2$ games] \label{ex:23generic}
The following two payoff matrices exhibit the generic behavior:
 \begin{equation}
 \label{eq:genericbehavior} 
 \begin{small}
 X^{(1)}\, =\,
 \begin{bmatrix}
  a_{11} & \! a_{12} \\
  a_{21} & \! a_{22} \\
  a_{31} & \! a_{32} \end{bmatrix} \,=\,
    \begin{bmatrix} 0 & 30 \\ 5 & 25 \\ 13 & 24   \end{bmatrix}
 \quad {\rm and} \quad \,\, X^{(2)} \,=\,
 \begin{bmatrix}
  b_{11} & \! b_{12} \\
  b_{21} & \! b_{22} \\
  b_{31} & \! b_{32} \end{bmatrix} \,=\,
    \begin{bmatrix}6 & 42 \\ 21 & 12 \\ 36 & 0  \end{bmatrix}.
 \end{small}
\end{equation}
The polygon $\pi_X(\Delta)$ is the pentagon whose vertices are $(a_{ij},b_{ij})$ with $\{i,j\} \neq \{2,2\}$. The payoff region $\mathcal{P}_X= \pi_X( \mathcal{V}_X \cap \Delta )$ is shaded in blue in
Figure \ref{fig: payoff region}.  The algebraic
 boundary of $\mathcal{P}_X$ is given by the two cubics 
 $9x_1^2x_2-2x_1x_2^2-162x_1^2-189x_1x_2+30x_2^2+3906x_1-540x_2+2160$ and $72x_1^2x_2-19x_1x_2^2-1512x_1^2-1614x_1x_2+390x_2^2+36288x_1-2340x_2$, plus
the two vertical lines $x_1-13$ and $x_1-24$.
The two curvy triangles that form $\mathcal{P}_X$ meet at the special point
\begin{equation}
\label{eq:specialpoint}
\bigl(22.9902299164, 16.2987107576 \bigr).
\end{equation}

Figure \ref{fig: payoff region} illustrates the general behavior for 
$3 \times 2$ games. We can understand this via the del Pezzo
geometry in Example~\ref{ex:delpezzo}.
The Spohn surface $\mathcal{V}_X$ is 
the blow-up of $\PP^1 \times \PP^1$ at six points. One of these six is the
special point (\ref{eq:specialpoint}). The
Konstanz matrix $K_X(x)$ in (\ref{eq:Konstanz32}) has rank four at this point,
so there is a line segment in $\mathcal{V}_X \cap \Delta$ that maps to (\ref{eq:specialpoint})
under $\pi_X$. At all nearby points $x\in \RR^2$, the rank of $K_X(x)$ is five. Here,
$\pi_X$ gives a bijection between $\mathcal{V}_X \cap \Delta$ and the payoff region 
$\mathcal{P}_X$.
The boundary curves of $\mathcal{P}_X$ are defined by 
maximal minors of $K_X(x)$.
Each minor is a $5 \times 5$-determinant, but it has degree four 
as a polynomial in $x = (x_1,x_2)$. That quartic factors into
a linear factor $x_1-a_{ij}$ times  a cubic  in $(x_1,x_2)$.
\hfill $\diamond$
\end{example}

We now work towards the main result of this section,
 generalizing Example \ref{ex:23generic} to arbitrary tensor formats.
The key players are the maximal minors of the Konstanz matrix $K_X(x)$.

\begin{lemma} \label{lem:minorsdegree}
Given any game $X$,
each of the $\binom{d_1 d_2 \cdots d_n}{\,d_1+d_2+ \cdots+d_n\,}$ maximal
minors of the Konstanz matrix $K_X(x)$ is a polynomial of degree at most
$\sum_{i=1}^n d_i -n + 1$ in the unknowns $x_1,\ldots,x_n$.
\end{lemma}

\begin{proof}
The highest degree seen in the maximal minors is the rank of $K_X(x)$
after setting all entries in the payoff tables $X^{(i)}$ to zero.
After rescaling the rows, the columns of this matrix are homogeneous coordinates
for the vertices of the product of standard simplices $\Delta_{d_1-1} \times \cdots \times \Delta_{d_n-1}$.
The dimension of this polytope is one less than the matrix rank.
\end{proof}

Suppose now that $X$ is fixed and generic.
We consider the stratification of the payoff space $\RR^n$ defined by the
signs taken on by the maximal minors of $K_X(x)$. 
We call this the {\em oriented matroid stratification}
of the game $X$. Indeed, it is the restriction to $\RR^n$ of the usual
oriented matroid stratification (cf.~\cite{Mnev})
of the space of matrices with $\sum_{i=1}^n d_i$ rows and $\prod_{i=1}^n d_i$ columns.
The maximal minors of $K_X(x)$ that are nonzero polynomials give the bases
of a matroid. The full-dimensional strata correspond to orientations of that matroid. The open stratum
containing a given point $x \in \RR^n$ consists of all points $x' \in \RR^n$ such that
corresponding nonzero maximal minors of $K_X(x)$ and $K_X(x')$ have the same sign $+1$ or $-1$.

The oriented matroid strata in $\RR^n$ are semialgebraic. Their boundaries 
are delineated by the maximal minors of $K_X(x)$.
These minors are the polynomials in Lemma~\ref{lem:minorsdegree}.
The oriented matroid strata can be disconnected (cf.~\cite{Mnev}). 
This happens in Examples \ref{ex:disconnected} and~\ref{ex:23generic}.
Note that the union of
the two open curvy triangles in Figure \ref{fig: payoff region} is a single chamber
(open stratum) for the  game $X$ given in  (\ref{eq:genericbehavior}).
It is given by prescribing a fixed sign $+1$ or $-1$ for each of the six maximal minors of
(\ref{eq:Konstanz32}). Interestingly, $\mathcal{P}_X$ itself is connected
in this case. The point  (\ref{eq:specialpoint}) lies in $\mathcal{P}_X$
because its fiber under $\pi_X$ is a line that meets the interior of $\Delta$.

We now present our characterization of the payoff region $\mathcal{P}_X$ of a generic game $X$.
By the {\em algebraic boundary} of $\mathcal{P}_X$
we mean the Zariski closure of its topological boundary.

\begin{theorem} \label{thm:OM}
The payoff region $\mathcal{P}_X$ for a generic game $X$ is a union of
oriented matroid strata in $\RR^n$ that are given by the
signs of the maximal minors of the Konstanz matrix $K_X(x)$.
Its algebraic boundary is a union of irreducible
hypersurfaces of degree at most $\,\sum_{i=1}^n d_i - n + 1$.
\end{theorem}

\begin{proof} For fixed $x \in \RR^n$, the set of probability tensors $P$ with 
expected payoffs $x$ is equal to
\begin{equation}
\label{eq:polytopefiber}
 {\rm kernel}\bigl(K_X(x)\bigr) \, \, \cap \,\, \Delta . 
 \end{equation}
 This is a convex polytope which is either empty or has the full
 dimension $\prod_{i=1}^n d_i - \sum_{i=1}^n d_i - 1$.
 The payoff region $\mathcal{P}_X$ is the set of
 all $x \in \RR^n$ such that this polytope is nonempty.
 We know from oriented matroid theory \cite[Chapter 9]{BLSWZ}
 that the combinatorial type of the polytope (\ref{eq:polytopefiber})
is determined by the oriented matroid of the matrix $K_X(x)$. 
Therefore, the combinatorial type is constant as $x$ ranges
over a fixed oriented matroid stratum in $\RR^n$.
In particular, whether or not (\ref{eq:polytopefiber}) is empty 
depends only on the oriented matroid of $K_X(x)$. Namely,
it is non-empty if and only if every column index lies in a positive
covector of that oriented matroid. This proves the first
sentence. The second sentence follows from 
 Lemma \ref{lem:minorsdegree}.
\end{proof}

One of the reasons for our interest in the algebraic boundary is that it
helps in characterizing dependency equilibria $P$ that are Pareto optimal. We thus 
address a question raised in \cite[Section 4]{spohn}.
Recall that $P$ is  {\em Pareto optimal} if its image $x = \pi_X(P)$ in $\mathcal{P}_X$ satisfies
$\,(x + \RR^n_{\geq 0})\, \cap \,\overline{\mathcal{P}_X} = \{x\}$.
This condition implies that $x$ lies in the boundary of $\mathcal{P}_X$, hence
one of the maximal minors of $K_X(x)$ must vanish. For instance,
for the $3 \times 2$ game in Example~\ref{ex:23generic}, the Pareto optimal equilibria
correspond to the points on the upper-right boundaries of the two curvy triangles
in Figure \ref{fig: payoff region}. At such points $x$, the product of our two cubics~vanishes.

We close this section by  discussing
Theorem \ref{thm:OM} for two cases larger than Example~\ref{ex:23generic}.

\begin{example}[$3\times3$ games] \label{ex:33Konstanz}
Let $n=2$ and $d_1=d_2 = 3$. The Konstanz matrix equals
$$
 K_X(x)  \,=\,  \begin{small}
\begin{bmatrix}
     x_1 \!-\! a_{11} & \! x_1 \!-\! a_{12} \! &\! x_1\! -\! a_{13} \!& 0 & 0 & 0 & 0 & 0 & 0 \\
   0 & 0 & 0 & x_1 \!-\! a_{21} & \! x_1 \!-\! a_{22}\! & \! x_1 \!-\! a_{23}\! & 0 & 0 & 0 \\
   0 & 0 & 0 & 0 & 0 & 0 & \! x_1 \!-\! a_{31}\!  & \! x_1 \!-\! a_{32}\! & \!x_1 \!-\! a_{33} \\
   x_2 \!-\! b_{11} & 0 & 0 & \! x_2 \!-\! b_{21} & 0 & 0 &\! x_2 \!-\! b_{31} & 0 & 0 \\
   0 & \! x_2 \!-\! b_{12} \! & 0 & 0 & \! x_2 \!-\! b_{22} \! & 0 & 0 & \! x_2\! -\! b_{32} \! & 0 \\
   0 & 0 & \! x_2 \!-\! b_{13}\!  & 0 & 0 & \! x_2 \!-\! b_{23} & 0 & 0 & \! x_2 \! - \! b_{33} 
\end{bmatrix}\!.
\end{small}
$$
Among the $\binom{9}{6} = 84$ maximal minors of this $6 \times 9$ matrix,
six are identically zero. Six others are irreducible polynomials of degree five in $x=(x_1,x_2)$.
Each of the remaining $72$ minors is an irreducible cubic times a product $(x_1-a_{ij})(x_2-b_{kl})$.
The resulting arrangement~of lines, cubics and quintics divides the plane $\RR^2$ into 
open chambers. We examine the chambers that lie inside
the polygon $\pi_X(\Delta)$. The rank $6$ oriented matroid of $K_X(x)$, given by $78$ signed bases,
 is constant on each chamber.  The payoff region is a union of some of them.
 \hfill $\diamond$
\end{example}

\begin{example}[$2\times 2 \times 2$ games] \label{ex:222Konstanz}
The game played by Adam, Bob and Carl in Example \ref{ex:twotwotwo}~has
$$
 K_X(x)  \,=\,  \begin{small}
\begin{bmatrix}
  x_1 \!-\! a_{111} & x_1 \!-\! a_{112} & x_1 \!-\! a_{121} & x_1 \!-\! a_{122} & 0 & 0 & 0 & 0 \\
0 & 0 & 0 & 0 & x_1 \!-\! a_{211} & x_1 \!-\! a_{212} & x_1 \!-\! a_{221} & x_1 \!-\! a_{222} \\
x_2 \!-\! b_{111} & x_2 \!-\! b_{112} & 0 & 0 & x_2 \!-\! b_{211} & x_2 \!-\! b_{212} & 0 & 0 \\
0 & 0 & x_2 \!-\! b_{121} & x_2 \!-\! b_{122} & 0 & 0 & x_2 \!-\! b_{221} & x_2 \!-\! b_{222} \\
x_3 \!-\! c_{111} & 0 & x_3 \!-\! c_{121} & 0 & x_3 \!-\! c_{211} & 0 & x_3 \!-\! c_{221} & 0 \\
0 & x_3 \!-\! c_{112} & 0 & x_3 \!-\! c_{122} & 0 & x_3 \!-\! c_{212} & 0 & x_3 \!-\! c_{222} 
\end{bmatrix}\! .
\end{small}
$$
All $\binom{8}{6} = 28$ maximal minors are irreducible polynomials of degree four in $x = (x_1,x_2,x_3)$.
Each of them defines a smooth quartic surface in $\CC^3$ that has three isolated singularities at
infinity in $\PP^3$. This data specifies 
an arrangement of $28$ K3 surfaces in $\PP^3$.
We examine its chambers inside
the polytope  $\pi_X(\Delta)$, which has $\leq 8$ vertices. The payoff region $\mathcal{P}_X$ is the union
of a subset of these chambers, so its algebraic boundary consists of quartic surfaces.
\hfill $\diamond$
\end{example}

\section{Conditional Independence and Bayesian Networks}
\label{sec6}

One drawback of dependency equilibria is that they are abundant.
Indeed, if the Spohn variety $\mathcal{V}_X$ 
intersects the open simplex $\Delta$, then the semialgebraic set 
$\mathcal{V}_X \cap \Delta$ of all dependency equilibria
has dimension $\prod_{i=1}^n d_i - \sum_{j=1}^n d_j + n - 1$.
This follows from Theorem~\ref{thm:spohn}.
To mitigate this drawback, we restrict to intersections of $\mathcal{V}_X$
with statistical models in $\Delta$. Natural candidates are
 the conditional independence models in 
 \cite[Section 8.1]{CBMS} and \cite[Section 4.1]{Sul}.

We view the $n$ players as random variables
with state spaces $[d_1],\ldots,[d_n]$. A point $P$ 
in $\Delta$ is a joint probability distribution.
Let $\mathcal{C}$ be any collection of conditional
independence (CI) statements on $[n]$. These statements have the form $A \! \perp \!\!\! \perp  \! B \,|\, C$, where
$A,B,C$ are pairwise disjoint subsets of $[n]$.  Each CI statement
  translates into a system of homogeneous quadratic constraints
in the tensor entries $p_{j_1 j_2 \cdots j_n}$.
This translation is explained in \cite[Proposition 8.1]{CBMS} 
and \cite[Proposition 4.1.6]{Sul}.
We write $\mathcal{M}_\mathcal{C}$ for the
projective variety in $\PP(V)$ that is defined by
these quadrics, arising from all statements 
$\,A \! \perp \!\!\! \perp \! B \,|\, C\,$ in $\mathcal{C}$.
Here we assume that components lying in the hyperplanes
$\{p_{j_1j_2 \cdots j_n}=0\}$ and $\{p_{++ \cdots +} = 0\}$ have been removed.

Suppose $X$ is any game in normal form, and $\mathcal{C}$ is any collection
of CI statements.
We define the {\em Spohn CI variety} to be the intersection
of the Spohn variety with the CI model:
\begin{equation}
\label{eq:spohnCI} \mathcal{V}_{X,\mathcal{C}} \,\, = \,\,
\mathcal{V}_X \, \cap \, \mathcal{M}_\mathcal{C}. 
\end{equation}
We again assume that components lying in the special hyperplanes above have been removed.
The intersection $\,\mathcal{V}_{X,\mathcal{C}} \,\cap \, \Delta\,$
with the simplex $\Delta$ is the set of all
 {\em CI equilibria} of the game $X$. 
 This is a semialgebraic set which is a natural extension of the
 set of Nash equilibria of $X$. 
  In what follows we assume that all random variables
 are binary, i.e.~$d_1=d_2 = \cdots = d_n = 2$.

  \begin{example}[Nash points] \label{ex:nashpoints}
 Let $\mathcal{C}$ be the set of all CI statements on $[n]$.
  The model $\mathcal{M}_\mathcal{C}$ is the Segre variety of rank one tensors, and
  the Spohn CI variety     (\ref{eq:spohnCI}) is the set of all Nash points in the Spohn variety $\mathcal{V}_X$.
  By \cite[Corollary 6.9]{CBMS}, this variety is finite, and its cardinality is the number of derangements of $[n]$,
  which is $1,2,9,44,265,\ldots$ for $n=1,2,3,4,5,\ldots$. 
  
  For $n \geq 3$, the Nash points span a linear subspace of codimension $2n$ in
$\PP(V) \simeq \PP^{2^n-1}$. To see this, we note that the
$i$th multilinear equation in \cite[Theorem 6.6]{CBMS} has degree $n-1$ and it
misses the $i$th unknown $\pi^{(i)}$.
Multiplying that equation by $\pi^{(i)}$ and by $1-\pi^{(i)}$
 gives two linear constraints on $\PP(V)$ for each $i$.
These $2n$ linear forms are linearly independent.
\hfill $\diamond$
 \end{example}

\begin{example}[$n=3, d_1=d_2=d_3=2$] \label{ex:3222}
Consider  games $X$ for three players with binary choices.
The Spohn variety $\mathcal{V}_X$ is a complete intersection
of dimension $4$ and degree $8$ in $\PP^7$. It is defined
by imposing rank one constraints on the three matrices $M_i$ in Example~\ref{ex:twotwotwo}.
It is parametrized by the lines ${\rm ker}(K_X(x))$
where $x \in \CC^3$ and $K_X(x)$ is the matrix in Example~\ref{ex:222Konstanz}.

We examine the Spohn CI varieties given by three models $\mathcal{M}_\mathcal{C}$
 in \cite[Section 8.1]{CBMS}. In each case, the intersection (\ref{eq:spohnCI})
is transversal in $\Delta$, and we find that $\mathcal{V}_{X,\mathcal{C}}$ is irreducible in~$\PP^7$.
\begin{itemize}
\item[(a)] Let $\mathcal{C} = \{\,1 \! \perp\!\!\! \perp \! 2 \,|\, 3\}$ as in \cite[eqn (8.3)]{CBMS}.
The CI model $\mathcal{M}_\mathcal{C}$ has codimension $2$ and degree $4$,
and the Spohn CI variety $\mathcal{V}_{X,\mathcal{C}}$ is a surface of degree
$28$ in $\PP^7$. We find that the prime ideal of $\mathcal{V}_{X,\mathcal{C}}$
 is minimally generated by five quadrics and three quartics.
\item[(b)] Let $\mathcal{C} = \{\,2 \! \perp\!\!\! \perp \! 3 \,\}$ as in \cite[eqn (8.4)]{CBMS},
so here $C = \emptyset$. 
The CI model $\mathcal{M}_\mathcal{C}$ is the hypersurface,
defined by the quadric $p_{+11} p_{+22} - p_{+12} p_{+21}$.
The Spohn CI variety $\mathcal{V}_{X,\mathcal{C}}$ is a threefold of degree
$10$ in~$\PP^7$. 
 Its prime ideal  is minimally generated by six quadrics.
\item[(c)]
Let $\mathcal{C} = \{\,1 \! \perp\!\!\! \perp \! 23 \,\}$ as in
\cite[eqn (8.5)]{CBMS}.
Here $\mathcal{M}_\mathcal{C} \simeq \PP^1 \times \PP^3$ is
defined by the $2 \times 2$ minors of a $2 \times 4$ matrix obtained by
flattening the tensor $P$. The  Spohn CI variety~$\mathcal{V}_{X,\mathcal{C}}$ 
is a curve of degree $8$ and genus $3$. It lies in a $\PP^5 $ inside $\PP^7$.
  Its prime ideal  is  generated by two linear forms and seven quadrics.
  These will be explained after Example \ref{ex:fifteen}.
  \end{itemize}
The computation of the prime ideals is non-trivial.
One starts with the ideal generated by the natural quadrics 
defining (\ref{eq:spohnCI}), and one then saturates that ideal by
$\,p_{+++} \cdot \prod_{i,j,k=1}^2  p_{ijk} $.
We performed these computations with the
computer algebra system {\tt Macaulay2} \cite{M2}.
\hfill $\diamond$
\end{example}

Of special interest are graphical models,
such as Markov random fields and Bayesian networks.
These allow us to describe the nature of the
desired equilibria by means of a graph whose
nodes are the $n$ players. This is different from
the setting of graphical games in \cite[Section 6.5]{CBMS},
where the graph structure imposes zero patterns in the
payoff tables $X^{(i)}$.

Inspired by \cite[Section 3]{homo}, we now focus on Bayesian networks, where
the CI statements
 $\mathcal{C}$ describe the
global Markov property of an acyclic directed graph with vertex set $[n]$.
These CI statements and their ideals are explained in
\cite[Section 3]{GSS}. In {\tt Macaulay2}, they can be 
computed using the commands {\tt globalMarkov} and {\tt conditionalIndependenceIdeal}
in the {\tt GraphicalModels} package.
Sometimes, it is preferable to work with the
prime ideal ${\rm ker}(\Phi)$ in \cite[Theorem 8]{GSS}.
From this we obtain the ideal of the
Spohn CI variety $\mathcal{V}_{X,\mathcal{C}}$ by
saturation, as described at the end of Example \ref{ex:3222}.
For all the models  we were able to compute, this ideal turned out to be of the expected codimension. 
In each case, except for the network with no edges, the variety $\mathcal{V}_{X,\mathcal{C}}$ is irreducible. We conjecture that these facts hold in general.

\begin{conjecture} \label{conj:smallbayes}
For every Bayesian network $\mathcal{C}$
on $n$ binary random variables, the Spohn CI variety $\,\mathcal{V}_{X,\mathcal{C}}\,$
 has the expected codimension $n$
inside the model $\mathcal{M}_\mathcal{C}$ in   $\PP^{2^n-1}$.
  The variety $\mathcal{V}_{X,\mathcal{C}}$ is positive-dimensional
  and irreducible whenever the network has at least one~edge.
  \end{conjecture}
  
  \begin{proposition} Conjecture \ref{conj:smallbayes} holds for $n \leq 3$.   \end{proposition}
  
\begin{proof}
For the network with no edges, $\mathcal{M}_\mathcal{C}$ is the
Segre variety $(\PP^1)^n$. The dimension statement holds, but the
Spohn CI variety is reducible, as seen in Example~\ref{ex:nashpoints}.
We thus examine all Bayesian networks with at least one edge. These satisfy
${\rm dim}(\mathcal{M}_\mathcal{C}) \geq n+1$.
The case $n \leq 2$ being trivial, we assume that $n=3$. If the network is a complete directed acylic graph, then the ideal of $\mathcal{M}_\mathcal{C}$ is the zero ideal and $\mathcal{V}_{X,\mathcal{C}} = \mathcal{V}_X$.
There are four networks left to be considered. By \cite[Proposition 5]{GSS},
 they are precisely the three models in Example~\ref{ex:3222}:
$$ \hbox{(a) $\,\,1 \leftarrow 3 \rightarrow 2\,\,$ or
 $\,\,2 \rightarrow 3 \rightarrow 1$ \qquad \quad
 (b) $\,\,3 \rightarrow 1 \leftarrow 2$ \qquad \quad
(c) $\,\,3 \rightarrow 2 \quad 1$.
} $$
This means that the proof was already given by our analysis in Example~\ref{ex:3222}.
\end{proof}

Consider the next case $n=4$. Up to relabeling, there are $29$ Bayesian networks $\mathcal{C}$ with at least one edge.
They are listed in \cite[Theorem 11]{GSS}, along with a detailed analysis  of the variety
 $\mathcal{M}_\mathcal{C}$ in each case. We embarked 
towards a proof of Conjecture \ref{conj:smallbayes}, by examining
all  $29$ models. But the computations
are quite challenging, and we leave them for the future.

\begin{example} \label{ex:fifteen}
Consider the network \#15 in \cite[Table 1]{GSS}. 
The variety $\mathcal{M}_\mathcal{C}$ has dimension $9$ and degree $48$.
An explicit parametrization $\phi$ is shown in \cite[page 109]{CBMS}.
We can represent $\mathcal{V}_{X,\mathcal{C}}$ by
substituting this parametrization  into the equations
${\rm det}(M_i) = 0$ for $i=1,2,3,4$.
\hfill $\diamond$
\end{example}

The smallest irreducible variety in Conjecture \ref{conj:smallbayes} 
arises from the Bayesian network $\mathcal{C}$ with only one edge, here taken to be  $n \rightarrow n-1$.
The Spohn CI variety $\mathcal{V}_{X,\mathcal{C}}$  contains all the Nash points
in Example  \ref{ex:nashpoints}.
The rest of this paper is dedicated to this scenario.
It is important for applications of
dependency equilibria because of its proximity to
Nash equilibria.

For our one-edge network,
 $\mathcal{M}_\mathcal{C} $ is the Segre variety 
$ (\PP^1)^{n-2} \times \PP^3$ embedded into $\PP^{2^n-1}$.
Hence $\mathcal{M}_\mathcal{C}$ has dimension $n+1$.
The Spohn CI variety  $\mathcal{V}_{X,\mathcal{C}}$ is a curve.
This curve  lies in a linear subspace of codimension  $2n-4$ in
$\PP^{2^n-1}$.  In addition to the quadrics that define the Segre variety  $\mathcal{M}_\mathcal{C} $,
the ideal of  $\mathcal{V}_{X,\mathcal{C}}$ contains
$2n-4$ linear forms and $2^{n-1}$ quadrics 
that depend on the game $X$.  The determinants of the matrices
$M_1,M_2,\ldots,M_{n-2}$ give rise to two linear forms each.
The determinants of the matrices $M_{n-1}$ or $ M_{n}$ give rise to
$2^{n-2}$ quadrics.

For example, if $n=3$ then the variety $\mathcal{M}_\mathcal{C} \simeq \PP^1 \times \PP^3$
has the parametric representation
$$ p_{ijk} \,\,= \,\,\sigma_i \tau_{jk} \quad \hbox{for} \,\,\,1 \leq i,j,k \leq 2. $$
The prime ideal of $\mathcal{M}_\mathcal{C}$  is generated by the six $2 \times 2 $ minors of the matrix
\begin{equation}
\label{eq:2by4} 
\begin{bmatrix} 
\,p_{111}  & p_{112} & p_{121} & p_{122} \\
\,p_{211}  & p_{212} & p_{221} & p_{222} 
\end{bmatrix}.
\end{equation}
After removing common factors from rows and columns, the three matrices 
in Example \ref{ex:twotwotwo}~are
$$ M_1 \,= \, \begin{bmatrix}
 \, \,\, 1 \, &\, \, a_{111} \tau_{11} + a_{112} \tau_{12} + a_{121} \tau_{21} + a_{122} \tau_{22} \\
\,  \,\, 1 \, &\,\,  a_{211} \tau_{11} + a_{212} \tau_{12} + a_{221} \tau_{21} + a_{222} \tau_{22}
\end{bmatrix},
$$
$$
M_2 \,= \,\begin{bmatrix}
   \,  \tau_{11} + \tau_{12}  &\,   b_{111} \sigma_{1} \tau_{11} 
        + b_{112} \sigma_{1} \tau_{12} + b_{211} \sigma_{2} \tau_{11} + b_{212} \sigma_{2} \tau_{12} \\
   \,  \tau_{21} + \tau_{22} &\,   b_{121} \sigma_{1} \tau_{21}
       + b_{122} \sigma_{1} \tau_{22} + b_{221} \sigma_{2} \tau_{21} + b_{222} \sigma_{2} \tau_{22}
\end{bmatrix},
$$      
$$
M_3 \,= \,\begin{bmatrix}
   \, \tau_{11} + \tau_{21} &\,   c_{111} \sigma_{1} \tau_{11} + c_{121} \sigma_{1} \tau_{21} 
                                        + c_{211} \sigma_{2} \tau_{11} + c_{221} \sigma_{2} \tau_{21} \\
  \, \tau_{12} + \tau_{22} & \,  c_{112} \sigma_{1} \tau_{12} + c_{122} \sigma_{1} \tau_{22}
                                       + c_{212} \sigma_{2} \tau_{12} + c_{222} \sigma_{2} \tau_{22}
\end{bmatrix}.
$$
By multiplying ${\rm det}(M_1)$ with $\sigma_1$ and with $\sigma_2$,
we obtain two linear forms in $p_{111},p_{112},\ldots,p_{222}$
that vanish on $\mathcal{V}_X$. Likewise, by multiplying 
${\rm det}(M_2)$ and ${\rm det}(M_3)$    with $\sigma_1$ and with $\sigma_2$,
we obtain four quadratic forms in $p_{111},p_{112},\ldots,p_{222}$
that vanish on $\mathcal{V}_X$.  Three of the six minors of (\ref{eq:2by4})
are linearly independent modulo the linear forms. This explains the
$2+7$ generators of the prime ideal of the curve $\mathcal{V}_{X,\mathcal{C}}$,
which has genus $3$ and degree $8$ in $\PP^5 \subset \PP^7$.

Let now $n=4$. The one-edge model
 $\mathcal{M}_\mathcal{C} $ is the Segre variety $\PP^1 \times \PP^1 \times \PP^3$ in 
$\PP^{15}$. Its prime ideal is generated by $46$ binomial quadrics.
Of these, $32$ are linearly independent modulo the four linear forms
that arise from the matrices $M_1$ and $M_2$
as above. Similarly,  $M_3$ and $M_4$ contribute eight quadrics.
 We conclude that $\mathcal{V}_{X,\mathcal{C}}$ is an
  curve of genus $23$ and degree $30$ in
$\PP^{11} \subset \PP^{15}$, and its prime ideal is minimally
generated by $4$ linear forms and $40$ quadrics.

In the recent work \cite{portakalarranz} it is proven, for generic games, that the Spohn CI curve for the one-edge model is an irreducible complete intersection curve in the Segre variety $(\mathbb{P}^{1})^{n-2} \times \mathbb{P}^3$. Moreover the authors give an explicit formula for its degree and genus. In the spirit of Datta's universality theorem for Nash equilibria, they show that any affine real algebraic variety $S \subseteq \mathbb{R}^m$ defined by $k$ polynomials with $k<m$ can be represented as the Spohn CI variety of an $n$-person game for one-edge Bayesian networks on $n$ binary random variables.
 \bigskip

\subsection*{Acknowledgements}
Both authors visited Konstanz in November 2021.
The ideas we encountered during our visit led to this article, and to  the name Konstanz matrix.
We thank Mantas Radzvilas, Gerard~Rothfus and Wolfgang Spohn for inspiring discussions about
dependency equilibria. We are grateful to
  Fulvio Gesmundo, Chiara Meroni,  Mateusz Micha{\l}ek and Kristian Ranestad for
  communications that greatly helped this project.
   Happy 60th Birthday to Giorgio Ottaviani, with mille grazie for being a fantastic
   teacher of applicable algebraic geometry.

 \bigskip \medskip

\noindent
{\bf Authors' addresses:}

\smallskip

\noindent Irem Portakal, Technical University of Munich
\hfill {\tt mail@irem-portakal.de}

\noindent Bernd Sturmfels,
MPI-MiS Leipzig and UC Berkeley
\hfill {\tt bernd@mis.mpg.de}
\end{document}